\documentclass[12pt]{amsart}
\usepackage{amsfonts}
\usepackage[colorlinks=true, pdfstartview=FitV, linkcolor=blue, citecolor=blue, urlcolor=blue]{hyperref}

\calclayout
\allowdisplaybreaks
\newtheorem{theorem}{Theorem}[section]

\newtheorem{lemma}[theorem]{Lemma}
\newtheorem{proposition}[theorem]{Proposition}

\theoremstyle{definition}

\theoremstyle{remark}
\newtheorem{remark}[theorem]{Remark}
\numberwithin{equation}{section}
\def\dint{\mathop{\displaystyle \int}}%
\def\func#1{\mathop{\rm #1}}

\def\dfrac#1#2{{\displaystyle {#1 \over #2}}}

\def\XXint#1#2#3{{\setbox0=\hbox{$#1{#2#3}{\int}$}
     \vcenter{\hbox{$#2#3$}}\kern-.5\wd0}}

\begin{document}
\date{June 7, 2009}
\title[The Kato problem]{The solution of the Kato problem for degenerate
elliptic operators with Gaussian bounds}
\author{David Cruz-Uribe, SFO}
\address{Trinity College, Hartford, CT, USA}
\email{david.cruzuribe@trincoll.edu}
\author{Cristian Rios}
\address{University of Calgary, Calgary, AB, Canada}
\email{crios@math.ucalgary.ca}
\subjclass{Primary 35J70, 35C15, 47D06 ; Secondary  47N20, 35K65}

\begin{abstract}
We prove the Kato conjecture for degenerate elliptic operators on ${\mathbb{R%
}^n}$. More precisely, we consider the divergence form operator ${\mathcal{L}%
}_w=-w^{-1} {\mathrm{div}}\mathbf{A}{\nabla} $, where $w$ is a Muckenhoupt $%
A_{2}$ weight and $\mathbf{A}$ is a complex-valued $n\times n$ matrix such
that $w^{-1}\mathbf{A}$ is bounded and uniformly elliptic. We show that if
the heat kernel of the associated semigroup $e^{-t{\mathcal{L}}_w}$
satisfies Gaussian bounds, then the weighted Kato square root estimate, $\|{%
\mathcal{L}}_w^{1/2} f\| _{L^{2}\left( w\right) }\approx \| {\nabla} f\|
_{L^{2}\left( w\right) }$, holds.
\end{abstract}

\thanks{The authors would like to thank Steven Hofmann for his suggestion
that we work on such a beautiful problem.}
\maketitle

\section{Introduction}

The purpose of this work is to give a positive answer to the Kato square
root problem for a class of degenerate elliptic operators, under the
assumption that the associated heat kernel satisfies classic Gaussian upper
bounds.

Before stating our results, we briefly sketch the background. Given a
uniformly elliptic, $n\times n$ complex matrix $\mathbf{A}$, define the
second-order elliptic operator ${\mathcal{L}}=-{\mathrm{div}}\mathbf{A}{%
\nabla }$. Then the square root ${\mathcal{L}}^{1/2}$ can be defined using
the functional calculus. The original Kato problem was to show that for all $%
f$ in the Sobolev space $H^{1}$, $\Vert {\mathcal{L}}^{1/2}f\Vert
_{L^{2}}\approx \Vert {\nabla }f\Vert _{L^{2}}$. This was first posed by
Kato~\cite{kato61} in 1961, but only solved in the past decade in a series
of remarkable papers by Auscher, \emph{et al.}~\cite%
{auscher-hofmann-lacey-mcintosh-tchamitchian02,
auscher-hofmann-lewis-tchamitchian01, hofmann-lacey-mcintosh02}. Initially,
they solved the problem given the additional assumption that the heat kernel
of the semigroup $e^{-t{\mathcal{L}}_{w}}$ satisfied Gaussian bounds. Such
estimates were known to be true in the case $\mathbf{A}$ was real symmetric,
but it had been shown that they need not hold for complex matrices in higher
dimensions~\cite{auscher-coulhon-tchamitchian96}. The final proof omitted
this hypothesis. For a more complete history of this problem, we refer the
reader to the above papers or to the review by Kenig~\cite{kenig04}.

We have extended this result to the case of degenerate elliptic operators,
where the degeneracy is controlled by a weight in the Muckenhoupt class $A_2$%
. We say that a weight $w$ (i.e., a non-negative, locally integrable
function) is in $A_2$ if 
\begin{equation*}
\sup_Q \left( 
\mathchoice {{\setbox0=\hbox{$\displaystyle{\textstyle
-}{\int}$} \vcenter{\hbox{$\textstyle -$}}\kern-.5\wd0}} {{\setbox0=\hbox{$\textstyle{\scriptstyle -}{\int}$} \vcenter{\hbox{$\scriptstyle -$}}\kern-.5\wd0}} {{\setbox0=\hbox{$\scriptstyle{\scriptscriptstyle -}{\int}$} \vcenter{\hbox{$\scriptscriptstyle -$}}\kern-.5\wd0}} {{\setbox0=\hbox{$\scriptscriptstyle{\scriptscriptstyle -}{\int}$} \vcenter{\hbox{$\scriptscriptstyle -$}}\kern-.5\wd0}}
\!\int_{Q}w\left( x\right) dx\right) \left( \mathchoice
{{\setbox0=\hbox{$\displaystyle{\textstyle -}{\int}$}
\vcenter{\hbox{$\textstyle -$}}\kern-.5\wd0}}
{{\setbox0=\hbox{$\textstyle{\scriptstyle -}{\int}$}
\vcenter{\hbox{$\scriptstyle -$}}\kern-.5\wd0}}
{{\setbox0=\hbox{$\scriptstyle{\scriptscriptstyle -}{\int}$}
\vcenter{\hbox{$\scriptscriptstyle -$}}\kern-.5\wd0}}
{{\setbox0=\hbox{$\scriptscriptstyle{\scriptscriptstyle -}{\int}$}
\vcenter{\hbox{$\scriptscriptstyle -$}}\kern-.5\wd0}} \!\int_{Q}w\left(
x\right) dx\right) =\left[ w\right] _{A_{2}}^{2}<\infty ,
\end{equation*}%
where the supremum is taken over all cubes $Q\subset \mathbb{R}^{n}$. Given $%
w\in A_2$ and constants $\lambda,\,\Lambda$, $0<\lambda\leq \Lambda<\infty$,
let ${\mathcal{E}}_n(w, \lambda, \Lambda )$ denote the class of $n\times n$
matrices $\mathbf{A}=\left( A_{ij}\left( x\right) \right) _{i,j=1}^{n}$ of
complex-valued, measurable functions satisfying the degenerate ellipticity
condition 
\begin{equation}  \label{eqn:degen}
\begin{cases}
\lambda w\left( x\right) | \xi | ^{2}\leq {\mathrm{Re}}\left\langle \mathbf{A%
}\xi ,\xi \right\rangle =\mathrm{{\mathrm{Re}}}{\sum_{i,j=1}^{n}}%
A_{ij}\left( x\right) \xi _{j}\bar{\xi}_{i}, \\ 
|\left\langle \mathbf{A}\xi ,\eta \right\rangle |\leq \Lambda w(x)|\xi
||\eta |%
\end{cases}%
\end{equation}%
for all $\xi ,\,\eta \in \mathbb{C}^{n}$.

Given $\mathbf{A}\in {\mathcal{E}}_n(w,\lambda,\Lambda)$, define the
degenerate elliptic operator in divergence form ${\mathcal{L}}_{w}=-w^{-1}{{%
\mathrm{div}}}\mathbf{A}{\nabla}$. Such operators were first considered by
Fabes, Kenig and Serapioni~\cite{fabes-kenig-serapioni82} and have been
considered by a number of other authors since. (See, for example, \cite%
{chen02,chen03,chen03b,chiarenza-franciosi87,
franchi91,kurata-sugano00,MR2382861}.) It is a natural question to extend
the Kato problem to these operators: that is, to show that 
\begin{equation*}
\|{\mathcal{L}}_w^{1/2}\|_{L^2(w)} \approx \|{\nabla} f\|_{L^2(w)}
\end{equation*}
for all $f$ in the weighted Sobolev space $H_0^1(w)$. (Exact definitions
will be given below.) We consider this problem in the special case that the
heat kernel of the semigroup $e^{-t{\mathcal{L}}_w}$ satisfies Gaussian
bounds. More precisely, we assume there exists a heat kernel $W_{t}(x,y)$
associated to the operator $e^{-t{\mathcal{L}}_{w}}$ such that for all $f\in
C_{c}^{\infty }$, 
\begin{equation}
e^{-t{\mathcal{L}}_{w}}f(x)=\int_{{\mathbb{R}^{n}}}W_{t}(x,y)f(y)\,dy.
\label{eqn:gauss1}
\end{equation}%
Furthermore, for all $t>0$ and $x,\,y\in {\mathbb{R}^{n}}$, the kernel $W_t $
satisfies the Gaussian bounds 
\begin{equation}
|W_{t}(x,y)|\leq \frac{C_{1}}{t^{n/2}}\exp \left( -C_{2}\frac{|x-y|^{2}}{t}%
\right) ,  \label{ineq:gauss2}
\end{equation}%
and the H\"{o}lder continuity estimates 
\begin{multline}  \label{ineq:gauss3}
| W_{t}\left( x+h,y\right) -W_{t}\left( x,y\right) | +| W_{t}\left(
x,y+h\right) -W_{t}\left( x,y\right) | \\
\leq \frac{C_{1}}{t^{n/2}}\left( \frac{| h| }{t^{1/2}+| x-y| }\right) ^{\mu
}\exp \left( -C_{2}\frac{| x-y| ^{2}}{t}\right) ,
\end{multline}
where $h\in \mathbb{R}^{n}$ is such that $2| h| \leq t^{1/2}+| x-y| $. The
constants $C_{1}$, $C_{2}$ and $\mu $ depend only on $n$, $w$, $\lambda $,
and $\Lambda $. If these three properties hold we will say that $e^{-t{%
\mathcal{L}}_w}$ satisfies Condition (G).

\medskip

Our main result is the following theorem.

\begin{theorem}
\label{theorem:WKato} Given $w\in A_2$ and $\mathbf{A}\in{\mathcal{E}}%
_n(w,\lambda,\Lambda)$, suppose that $e^{-t{\mathcal{L}}_w}$ satisfies
Condition (G). Then there exists a positive constant $C=C\left( n,\lambda
,\Lambda, w\right)$ such that for all $f$ in $H_{0}^{1}(w)$, 
\begin{equation}  \label{eqn:kato}
\| {\nabla} f\| _{L^2(w)}C^{-1}\leq \|{\mathcal{L}}_w^{1/2}f\| _{L^2(w)}\leq
C\| {\nabla} f\| _{L^2(w)}.
\end{equation}
\end{theorem}

To prove Theorem~\ref{theorem:WKato} it actually suffices to prove the
second inequality, 
\begin{equation}  \label{eqn:kato-half}
\|{\mathcal{L}}_w^{1/2}f\| _{L^2(w)}\leq C\| {\nabla} f\| _{L^2(w)}.
\end{equation}
For suppose this inequality holds. Since $\mathbf{A}\in \mathcal{E}\left(
n,\lambda ,\Lambda ,w\right) $ implies $\mathbf{A}^*\in \mathcal{E}\left(
n,\lambda ,\Lambda ,w\right) $, \eqref{eqn:kato-half} holds for $({\mathcal{L%
}}_w^{1/2})^* = ({\mathcal{L}}_w^*)^{1/2}$. (This operator identity follows
from the functional calculus, for instance, from \eqref{eqn:sqr-func-calc}.)
Therefore, by the ellipticity conditions 
\begin{align*}
\| {\nabla} f\| _{L^2(w)}^{2} = & \int_{\mathbb{R}^n} | {\nabla}
f(x)|^{2}w(x)\, dx \\
& \leq \lambda ^{-1} {\mathrm{Re}}\int_{\mathbb{R}^n} \mathbf{A}{\nabla}
f(x)\cdot \overline{{\nabla} f(x)}\,dx \\
= & \lambda ^{-1}{\mathrm{Re}}\int_{\mathbb{R}^n} {\mathcal{L}}_wf(x)%
\overline{f(x)}w(x)\,dx \\
= & \lambda ^{-1}{\mathrm{Re}}\int_{\mathbb{R}^n} {\mathcal{L}}_w^{1/2}f(x)%
\overline{ ({\mathcal{L}}_w^{1/2})^{\ast }f(x)} w(x)\, dx \\
\leq & \lambda ^{-1}\|{\mathcal{L}}_w^{1/2}f\|_{L^{2}\left( w\right) } \|({%
\mathcal{L}}_w^{*})^{1/2} f\| _{L^2(w)} \\
\leq & C \lambda ^{-1}\|{\mathcal{L}}_w^{1/2}f\| _{L^2(w)}\| {\nabla} f\|
_{L^{2}(w)}.
\end{align*}

\medskip

Our proof of inequality \eqref{eqn:kato-half} follows the outline of the
proof of the classical Kato problem with Gaussian bounds in \cite%
{hofmann-lacey-mcintosh02}. (See also the expository treatment in \cite%
{hofmann01}.) There are four main steps: in~Section~\ref%
{section:square-functions} we reduce \eqref{eqn:kato-half} to a square
function inequality; in Section~\ref{section:Carleson} we show that this
inequality is a consequence of a Carleson measure estimate; in Section~\ref%
{section:Tb} we prove a weighted $Tb$-theorem for square roots; finally, in
Section~\ref{section:test-func} we construct the family of test functions
need to use the $Tb$ theorem to prove the Carleson measure estimate. Prior
to the proof itself, in Sections~\ref{section:preliminaries} and~\ref%
{section:weights} we give some preliminary results about degenerate elliptic
operators, Gaussian bounds and weighted norm inequalities. And in Section~%
\ref{section:two-ineq} we prove two weighted square function inequalities
needed in our proof. The first, in particular, is central, since it is the
replacement for the (much simpler) Fourier transform estimates used in the
unweighted case.

Throughout, all notation is standard or will be defined as needed. The
letters $C$, $c$ will denote constants whose value may change at each
appearance. Given a function $f$ and $t>0$, define $f_{t}(x)=t^{-n}f(x/t)$.
Given an operator $T$ on a Banach space $X$, let $\Vert T\Vert _{{\mathcal{B}%
}(X)}$ denote the operator norm of $T$.

\section{Degenerate Elliptic Operators}

\label{section:preliminaries}

The properties of the degenerate elliptic operator ${\mathcal{L}}_w$ and the
associated semigroup $e^{-t{\mathcal{L}}_w}$ were developed in detail in~%
\cite{cruz-uribe-riosP} and we refer the reader there for complete details.
Here we state the key ideas.

Given a weight $w\in A_2$, the space $H_0^1(w)$ is the weighted Sobolev
space that is the completion of $C_c^\infty$ with respect to the norm 
\begin{equation*}
\|f\|_{H_0^1(w)} = \left(\int_{\mathbb{R}^n} \left(|f(x)|^2+|\nabla
f(x)|^2\right)w(x)\,dx\right)^{1/2}.
\end{equation*}
Given a matrix $\mathbf{A}\in {\mathcal{E}}_n(w,\lambda,\Lambda)$, define ${%
\mathfrak{a}}(f,g)$ to be the sesquilinear form 
\begin{equation}  \label{eqn-form}
{\mathfrak{a}}(f,g) = \int_{\mathbb{R}^n} \mathbf{A}(x)\nabla f(x) \cdot 
\overline{\nabla g(x)} \,dx.
\end{equation}
Since $w\in A_2$ and $\mathbf{A}$ satisfies \eqref{eqn:degen}, $\mathfrak{a} 
$ is a closed, maximally accretive, continuous sesquilinear form. Therefore,
there exists a densely defined operator ${\mathcal{L}}_w$ on $L^2(w)$ such
that for every $f$ in the domain of ${\mathcal{L}}_w$ and every $g\in L^2(w)$%
, 
\begin{equation}  \label{eqn-a2}
{\mathfrak{a}}(f,g) = \langle {\mathcal{L}}_wf, g \rangle_w= \int_{\mathbb{R}%
^n} {\mathcal{L}}_wf(x)\overline{g(x)}w(x)\,dx.
\end{equation}
If $f,\,g$ are in ${\mathcal{C}}_c^\infty$ and in the domain of ${\mathcal{L}%
}_w$, then integration by parts yields 
\begin{equation*}
\langle {\mathcal{L}}_wf, g \rangle_w = {\mathfrak{a}}(f,g) = \langle {%
\mathcal{L}} f,g\rangle,
\end{equation*}
where $\langle\,,\,\rangle$ is the standard complex inner-product on $L^2$.
Thus, at least formally, ${\mathcal{L}}_w = w^{-1}{\mathcal{L}}= -w^{-1}{%
\mathrm{div}}\mathbf{A}{\nabla}$.

Further, the properties of the sesquilinear form $\mathfrak{a}$ guarantee
that the semigroup $e^{-t{\mathcal{L}}_w}$ exists. In the special case when $%
\mathbf{A}$ is real and symmetric, then the heat kernel of $e^{-t{\mathcal{L}%
}_w}$ satisfies Condition (G).

Finally, in \cite{cruz-uribe-riosP} we proved the following results which
will be needed in our proof.

\begin{lemma}
\label{lemma:conservation} Given a matrix $\mathbf{A}\in {\mathcal{E}}%
_{n}(w,\lambda ,\Lambda )$, suppose that the heat kernel of the associated
semigroup $e^{-t{\mathcal{L}}_{w}}$ satisfies Condition (G). Then for all $%
t>0$, $e^{-t{\mathcal{L}}_{w}}1=1$: that is, for all $x\in {\mathbb{R}^n}$, 
\begin{equation*}
\int_{\mathbb{R}^n} W_t(x,y)\,dy = 1.
\end{equation*}
\end{lemma}

If $e^{-t{\mathcal{L}}_w}$ satisfies Gaussian bounds, then its derivative
satisfies similar bounds. More precisely, let $V_{t}=-2t{\mathcal{L}}%
_{w}e^{-t^{2}{\mathcal{L}}_{w}}=\frac{d}{dt}e^{-t^{2}{\mathcal{L}}_{w}}$.
Then the following result holds.

\begin{lemma}
\label{lemma:gaussian-prime} The operator $V_{t}$ has a kernel $V_{t}(x,y)$
with the following properties:

For all $x,\,y\in {\mathbb{R}^{n}}$ and $t>0$, 
\begin{equation}
|V_{t}(x,y)|\leq \frac{C_{1}}{t^{n+1}}\exp \left( -C_{2}\frac{|x-y|^{2}}{%
t^{2}}\right).  \label{ineq:gauss-p}
\end{equation}

For almost every $x\in {\mathbb{R}^{n}}$, $V_t1=0$, that is, for all $x\in {%
\mathbb{R}^n}$, 
\begin{equation*}
\int_{\mathbb{R}^{n}}V_{t}(x,y)\,dy=0.
\end{equation*}

There exists $\alpha =\alpha \left( n,\lambda ,\Lambda ,w\right) >0$ such
that for almost every $x,\,y,\,\in {\mathbb{R}^{n}}$, $2\left\vert
h\right\vert <t+\left\vert x\right\vert ,$ 
\begin{equation*}
|V_{t}(x,y)-V_{t}(x,y+h)|+|V_{t}(x+h,y)-V_{t}(x,y)|\leq \frac{C_{1}}{t^{n+1}}%
\left( \dfrac{\left\vert h\right\vert }{t+\left\vert x\right\vert }\right)
^{\alpha }\exp \left( -C_{2}\frac{|x-y|^{2}}{t^{2}}\right) .
\end{equation*}
\end{lemma}

\section{Weighted Norm Inequalities}

\label{section:weights}

Central to our proof is the theory of weighted norm inequalities for
classical operators, particularly singular integrals and square functions.
In this section we state the results we need; the standard ones are given
without proof and we refer the reader to Duoandikoetxea~\cite%
{duoandikoetxea01}, Garc\'{\i}a-Cuerva and Rubio de Francia~\cite%
{garciacuerva-rubiodefrancia85}, and Grafakos~\cite{grafakos04} for complete
information.

We begin with the weighted norm inequalities for the Hardy-Littlewood
maximal operator, for convolution operators, and for singular integrals.

\begin{lemma}
\label{prop:conv} Let $w\in A_2$. Then $M$ is bounded on $L^2(w)$ and $%
\|M\|_{\mathbf{B}(L^2(w))} \leq C(n,[w]_{A_2})$. Furthermore, suppose $\phi$
and $\Phi$ are such that for all $x$, $|\phi(x)|\leq \Phi(x)$, and $\Phi$ is
radial, decreasing and integrable. Then the operators $\phi_t*f$ are
uniformly bounded on $L^2(w)$; in fact, 
\begin{equation*}
\sup_{t>0} |\phi_t*f(x)|\leq C(n) \|\Phi\|_1 Mf(x).
\end{equation*}
\end{lemma}

\begin{remark}
\label{rem:conv} If $\phi$ is a Schwartz function then such a function $\Phi$
always exists.
\end{remark}

\medskip

\begin{lemma}
\label{prop:sio} Let $K : {\mathbb{R}^n}\setminus\{0\} \rightarrow {\mathbb{R%
}}$ be a locally integrable function such that $\widehat{K} \in L^\infty$
and 
\begin{equation*}
|{\nabla} K(x) | \leq \frac{C}{|x|^{n+1}}, \quad x\neq 0.
\end{equation*}
Then the singular integral 
\begin{equation*}
Tf(x) = \int_{\mathbb{R}^n} K(x-y)f(y)\,dy
\end{equation*}
and the maximal singular integral 
\begin{equation*}
T^*f(x) = \sup_{\epsilon>0}\left|\int_{|x-y|>\epsilon} K(x-y)f(y)\,dy\right|
\end{equation*}
are bounded on $L^2(w)$, and $\|T\|_{\mathbf{B}(L^2(w))},\,\|T^*\|_{\mathbf{B%
}(L^2(w))}\leq C(n,[w]_{A_2},K)$.
\end{lemma}

\begin{remark}
\label{rem:riesz} An important example of singular integrals are the Riesz
transforms: 
\begin{equation*}
R_jf(x) = c_n\int_{\mathbb{R}^n} \frac{x_j-y_j}{|x-y|^{n+1}}f(y)\,dy, \quad
1\leq j \leq n,
\end{equation*}
where the constant $c_n$ is chosen so that 
\begin{equation*}
\widehat{R_j f}(\xi) = -i \frac{\xi_j}{|\xi|}\widehat{f}(\xi).
\end{equation*}
\end{remark}

\medskip

Our next two results are square function inequalities. The first is a
weighted version of Carleson's theorem due to Journ\'e~\cite{journe83}.
Define the weighted Carleson measure norm of a function $\gamma _{t}$ by 
\begin{equation}
\Vert \gamma _{t}\Vert _{C,w}=\sup_{Q}\frac{1}{w(Q)}\int_{Q}\int_{0}^{\ell
(Q)}|\gamma _{t}(x)|^{2}\,\frac{dt}{t}\,w(x)\,dx.  \label{eqn:wtd-carleson}
\end{equation}

\begin{lemma}
\label{lemma:wtd-carleson} Let $w\in A_{2}$ and suppose $\gamma _{t}$ is
such that $\Vert \gamma _{t}\Vert _{C,w}<\infty $. Let $p\in C_{c}^{\infty }(%
{\mathbb{R}^{n}})$ be such that $p$ is a non-negative, radial, decreasing
function, ${\mathrm{supp}}(p)\subset B_{1}(0)$, and $\Vert p\Vert _{1}=1$.
Then for all $f\in L^{2}(w)$, 
\begin{equation*}
\int_{0}^{\infty }\int_{\mathbb{R}^{n}}|(p_{t}\ast f)(x)|^{2}|\gamma
_{t}(x)|^{2}\,w(x)\,dx\,\frac{dt}{t}\leq C\Vert \gamma _{t}\Vert _{C,w}\Vert
f\Vert _{L^{2}(w)}^{2}.
\end{equation*}
\end{lemma}

The second result is a Littlewood-Paley type inequality.

\begin{lemma}
\label{lemma:gfunction} Given $w\in A_2$, let $\psi$ be a Schwartz function
such that $\widehat\psi(0)=0$. Then for all $f\in L^2(w)$, 
\begin{equation}  \label{eqn:gfunction1}
\int_{\mathbb{R}^n} \int_0^\infty |\psi_t*f(x)|^2\,\frac{dt}{t}\,w(x)\,dx
\leq C(n,\psi,[w]_{A_2})\|f\|_{L^2(w)}^2.
\end{equation}
\end{lemma}

\begin{proof}
A direct proof of Lemma \ref{lemma:gfunction} is given by Wilson~\cite%
{wilsonP}. Here we sketch a proof that is implicitly based on the idea that%
\begin{equation*}
g_{\psi }(f)(x)=\left( \int_{0}^{\infty }|\psi _{t}\ast f(x)|^{2}\,\frac{dt}{%
t}\right) ^{1/2}
\end{equation*}
can be regarded as a vector-valued singular integral. 

By a standard argument in the theory of weighted norm inequalities, it will
suffice to prove that for all $0<\delta <1$ there exists a constant $%
C_{\delta }$ such that for $x\in {\mathbb{R}^{n}}$, 
\begin{equation}
M^{\#}(g_{\psi }(f)^{\delta })(x)\leq C_{\delta }Mf(x)^{\delta },
\label{eqn:sharp}
\end{equation}%
where $M$ is the Hardy-Littlewood maximal operator and $M^{\#}$ is the sharp
maximal operator of Fefferman and Stein.

The proof of inequality \eqref{eqn:sharp} is readily gotten by adapting the
argument in Cruz-Uribe and P\'{e}rez \cite[Lemma~1.6]{cruzuribe-perez02} for
the $g_{\lambda }^{\ast }$ operator. The changes are straightforward, so
here we only indicate the key steps. (Also see \'{A}lvarez and P\'{e}rez 
\cite{alvarez-perez94}.) Our assumptions on $\psi $ guarantee that $g_{\psi }
$ is bounded on $L^{2}$ and is weak $(1,1)$. (See \cite[p.~505]%
{garciacuerva-rubiodefrancia85}.) Therefore, we only have to prove that if $%
|x|>|h|/2$, then 
\begin{equation}
\left( \int_{0}^{\infty }|\psi _{t}(x+h)-\psi _{t}(x)|^{2}\frac{dt}{t}%
\right) ^{1/2}\leq C\frac{|h|^{1/2}}{|x|^{n+1/2}}.  \label{eqn:gfunction2}
\end{equation}
This is the vector-valued analog of the gradient condition in \cite[Lemma~1.6%
]{cruzuribe-perez02}. To prove \eqref{eqn:gfunction2}, note that since $\psi 
$ is a Schwartz function it is bounded. Hence, by the mean value theorem,
for each $x$ and $t$ there exists $\theta $, $0<\theta <1$, such that 
\begin{align*}
\int_{0}^{\infty }|\psi _{t}(x+h)-\psi _{t}(x)|^{2}\frac{dt}{t}&
=\int_{0}^{\infty }|\psi \left( \frac{x+h}{t}\right) -\psi \left( \frac{x}{t}%
\right) |^{2}\frac{dt}{t^{2n+1}} \\
& \leq C\int_{0}^{\infty }|\psi \left( \frac{x+h}{t}\right) -\psi \left( 
\frac{x}{t}\right) |\frac{dt}{t^{2n+1}} \\
& \leq C|h|\int_{0}^{\infty }|{\nabla }\psi \left( \frac{x+\theta h}{t}%
\right) |\frac{dt}{t^{2n+2}} \\
& =C|h|\int_{0}^{|x|}+C|h|\int_{|x|}^{\infty }
\end{align*}%
Since $|x+\theta h|>|x|/2$ and $|{\nabla }\psi (y)|\leq C|y|^{-2n-2}$, 
\begin{equation*}
|h|\int_{0}^{|x|}|{\nabla }\psi \left( \frac{x+\theta h}{t}\right) |\frac{dt%
}{t^{2n+2}}\leq C|h|\int_{0}^{|x|}|x|^{-2n-2}\,dt=C\frac{|h|}{|x|^{2n+1}}.
\end{equation*}%
We estimate the second integral in the same way, using that $|{\nabla }\psi
(x)|\leq C|x|^{-2n}$. Taking the square root we get \eqref{eqn:gfunction2}.
\end{proof}

\medskip

The next proposition is a key estimate in our proof of Theorem~\ref%
{theorem:WKato}. It yields a square function estimate given size and
regularity assumptions on the kernel of the operator. In the unweighted
case, this result can be found in, for example, Grafakos~\cite[p.~643]%
{grafakos04} or Hofmann~\cite{hofmann01}; in a somewhat different form it
can be found in Auscher and Tchamitchian~\cite{auscher-tchamitchian98}.

\begin{proposition}
\label{prop:schur} Let $w\in A_2$ and let $\psi$ be a radial Schwartz
function such that $\widehat{\psi}(0)=0$ and 
\begin{equation}  \label{eqn:repro1}
\int_0^\infty \widehat{\psi}(t)^2\,\frac{dt}{t} = 1.
\end{equation}
Let $Q_tf(x) = \psi_t*f(x)$. Given a family of sublinear operators $\{R_t\}$%
, suppose that each $R_t$ is bounded on $L^2(w)$, and for all $t,\,s>0$ the
composition $R_tQ_s$ is bounded on $L^2(w)$ and for some $\alpha >0$, 
\begin{equation}  \label{eqn:alpha}
\| R_{t}Q_{s}\|_{{\mathcal{B}}(L^2(w))}\leq K\min \left( \frac{t}{s},\frac{s%
}{t}\right)^{\alpha }.
\end{equation}
Then the family $\{R_{t}\}$ satisfies the square function estimate 
\begin{equation}  \label{eqn:schur}
\int_{0}^{\infty }\int_{{\mathbb{R}^n}}| R_tf(x)|^{2}w(x)\,dx\frac{dt}{t}
\leq K\cdot C(n,\psi,\alpha,[w]_{A_2})\|f\|_{L^2(w)}^2.
\end{equation}
\end{proposition}

The proof of Proposition \ref{prop:schur} requires a weighted version of the
Calder\'on reproducing formula given by Wilson~\cite{wilsonP}.

\begin{lemma}
\label{lemma:reproducing} For all $w\in A_2$ and $f\in L^2(w)$, 
\begin{equation*}
\int_0^\infty Q_t^2f(x)\frac{dt}{t} = f(x),
\end{equation*}
where this equality is understood as follows: for each $j>1$, let $B_j$ be
the ball centered at $0$ of radius $j$, and define the function 
\begin{equation*}
f_j(x) = \int_{1/j}^j Q_t(\chi_{B_j}Q_tf)(x) \frac{dt}{t}.
\end{equation*}
Then for each $j$, $f_j \in L^2(w)$ and $\{f_j\}$ converges to $f$ in $%
L^p(w) $.
\end{lemma}

\begin{proof}[Proof of Proposition~\textup{\protect\ref{prop:schur}}]
Fix $f\in L^{2}(w)$ and let $f_{j}$ be as in Lemma~\ref{lemma:reproducing}.
Since $R_{t}$ is bounded on $L^{2}(w)$, we have that for each $t>0$, 
\begin{equation*}
\int_{{\mathbb{R}^{n}}}|R_{t}f(x)|^{2}w(x)\,dx=\lim_{j\rightarrow \infty
}\int_{{\mathbb{R}^{n}}}|R_{t}f_{j}(x)|^{2}w(x)\,dx.
\end{equation*}%
Since each $R_{t}$ is sublinear, we have that 
\begin{equation*}
|R_{t}f_{j}(x)|\leq \int_{1/j}^{j}|R_{t}Q_{s}(\chi _{B_{j}}Q_{s}f)(x)|\frac{%
ds}{s}.
\end{equation*}%
Therefore, by Fatou's lemma, Minkowski's inequality, and \eqref{eqn:alpha}, 
\begin{align*}
& \int_{0}^{\infty }\int_{{\mathbb{R}^{n}}}|R_{t}f(x)|^{2}w(x)\,dx\,\frac{dt%
}{t} \\
& \qquad \leq \liminf_{j\rightarrow \infty }\int_{0}^{\infty }\int_{{\mathbb{%
R}^{n}}}|R_{t}f_{j}(x)|^{2}w(x)\,dx\,\frac{dt}{t} \\
& \qquad \leq \liminf_{j\rightarrow \infty }\int_{0}^{\infty }\int_{{\mathbb{%
R}^{n}}}\left( \int_{1/j}^{j}| R_{t}Q_{s}(\chi _{B_{j}}Q_{s}f)(x)| \frac{ds}{%
s}\right) ^{2}w(x)\,dx\,\frac{dt}{t} \\
& \qquad \leq \liminf_{j\rightarrow \infty }\int_{0}^{\infty }\left(
\int_{1/j}^{j}\left( \int_{\mathbb{R}^{n}}|R_{t}Q_{s}(\chi
_{B_{j}}Q_{s}f)(x)|^{2}w(x)\,dx\right) ^{1/2}\frac{ds}{s}\right) ^{2}\frac{dt%
}{t} \\
& \qquad \leq \liminf_{j\rightarrow \infty }K\int_{0}^{\infty }\left(
\int_{1/j}^{j}\min \left( \frac{t}{s},\frac{s}{t}\right) ^{\alpha }\Vert
\chi _{B_{j}}Q_{s}f\Vert _{L^{2}(w)}\frac{ds}{s}\right) ^{2}\frac{dt}{t} \\
& \qquad \leq K\int_{0}^{\infty }\left( \int_{0}^{\infty }\min \left( \frac{t%
}{s},\frac{s}{t}\right) ^{\alpha }\Vert Q_{s}f\Vert _{L^{2}(w)}\frac{ds}{s}%
\right) ^{2}\frac{dt}{t}.
\end{align*}%
For all $s>0$, 
\begin{equation*}
\int_{0}^{\infty }\min \left( \frac{t}{s},\frac{s}{t}\right) ^{\alpha }\frac{%
dt}{t}=\int_{0}^{\infty }\min \left( u,\frac{1}{u}\right) ^{\alpha }\frac{du%
}{u}=C(\alpha )<\infty ,
\end{equation*}%
and the same is true if we reverse the roles of $s$ and $t$. Therefore, if
we apply Schwartz' inequality, Fubini's theorem and Lemma~\ref%
{lemma:gfunction} we get that 
\begin{align*}
& \int_{0}^{\infty }\left( \int_{0}^{\infty }\min \left( \frac{t}{s},\frac{s%
}{t}\right) ^{\alpha }\Vert Q_{s}f\Vert _{L^{2}(w)}\frac{ds}{s}\right) ^{2}%
\frac{dt}{t} \\
& \qquad \leq C(\alpha )\int_{0}^{\infty }\left( \int_{0}^{\infty }\min
\left( \frac{t}{s},\frac{s}{t}\right) ^{\alpha }\frac{ds}{s}\right) \left(
\int_{0}^{\infty }\min \left( \frac{t}{s},\frac{s}{t}\right) ^{\alpha }\Vert
Q_{s}f\Vert _{L^{2}(w)}^{2}\frac{ds}{s}\right) \frac{dt}{t} \\
& \qquad \leq C(\alpha )\int_{0}^{\infty }\left( \int_{0}^{\infty }\min
\left( \frac{t}{s},\frac{s}{t}\right) ^{\alpha }\frac{dt}{t}\right) \Vert
Q_{s}f\Vert _{L^{2}(w)}^{2}\frac{ds}{s} \\
& \qquad \leq C(\alpha ,\psi ,n,[w]_{A_{2}})\Vert f\Vert _{L^{2}(w)}^{2}.
\end{align*}
\end{proof}

Operator norm bounds such as those in~\eqref{eqn:alpha} can generally be
deduced in the unweighted case using the Fourier transform or kernel
estimates. We will make use of the following result from Grafakos~\cite[%
Theorem~8.6.3]{grafakos04}.

\begin{lemma}
\label{lemma:reg-family} Let $\{T_t\}$, $t>0$ be a family of integral
operators such that $T_t1=0$ and such that the kernels $K_t$ satisfy 
\begin{gather}
|K_t(x,y)| \leq \frac{C}{t^n(1+t^{-1}|x-y|)^{n+1}},  \label{eqn:grafakos1} \\
|K_t(x,y)-K_t(x,y^{\prime })| \leq \frac{C|y-y^{\prime }|}{t^{n+1}}.
\label{eqn:grafakos2}
\end{gather}
Then for some $\alpha>0$, 
\begin{equation*}
\|T_t\|_{\mathbf{B}(L^2)} \leq C\min\left(\frac{t}{s},\frac{s}{t}%
\right)^\alpha.
\end{equation*}
\end{lemma}

In order to get this estimate on $L^2(w)$, $w\in A_2$, we use the following
clever application of interpolation due to Duoandikoetxea and Rubio de
Francia~\cite{duoandikoetxea-rubiodefrancia86}.

\begin{lemma}
\label{lemma:interpol} Suppose that a sublinear operator $T$ is bounded on $%
L^{2}(w)$ for all $w\in A_{2}$, with $\Vert T\Vert _{{\mathcal{B}}(L^{2}(w))}
$ depending only on $[w]_{A_{2}}$ and dimension $n$. Then for any $w\in A_{2}
$, there exists $\theta $, $0<\theta <1$, that depends on $[w]_{A_{2}}$ such
that, 
\begin{equation*}
\Vert T\Vert _{{\mathcal{B}}(L^{2}(w))}\leq C(n,[w]_{A_{2}})\Vert T\Vert _{{%
\mathcal{B}}(L^{2})}^{\theta }.
\end{equation*}
\end{lemma}

\begin{proof}
This is a consequence of the structural properties of $A_{2}$ weights and
the theory of interpolation with change of measure due to Stein and Weiss~ 
\cite{stein-weiss58} (see also Bergh and L\"{o}fstr\"{o}m~\cite%
{berg-lofstrom76}). Given $w\in A_{2}$, there exists $s>1$, depending only
on $[w]_{A_{2}}$, such that $w^{s}\in A_{2}$ with $[w^{s}]_{A_{2}}$
depending only on $[w]_{A_{2}}$. Hence, by hypothesis $\Vert T\Vert _{{%
\mathcal{B}}(L^{2}(w^{s}))}$ is bounded by a constant that depends only on $%
[w]_{A_{2}}$ and $n$. Choose $\theta $ such that $1-\theta =1/s$. Then $%
w=1^{\theta }w^{s(1-\theta )}$, so by interpolation with change of measure, 
\begin{equation*}
\Vert T\Vert _{{\mathcal{B}}(L^{2}(w))}\leq \Vert T\Vert _{{\mathcal{B}}%
(L^{2}(w^{s}))}^{1-\theta }\Vert T\Vert _{{\mathcal{B}}(L^{2})}^{\theta
}\leq C(n,[w]_{A_{2}})\Vert T\Vert _{{\mathcal{B}}(L^{2})}^{\theta }.
\end{equation*}
\end{proof}

\section{Two Square Function Inequalities}

\label{section:two-ineq}

In this section we prove two weighted square functions inequalities. The
first is for the operator $V_{t}=\frac{d}{dt}e^{-t^{2}{\mathcal{L}}_{w}}$
and is used in Sections~\ref{section:Carleson} and~\ref{section:Tb} below.

\begin{lemma}
\label{lemma-sq-Taylor}Let $p\in C_{c}^{\infty }$ be a non-negative, radial,
decreasing function such that $\Vert p\Vert _{1}=1$ and $\mathrm{{\mathrm{%
supp}}}(p)\subset B_{1}(0)$, and for $f\in H^{1}\left( w\right) $ let%
\begin{equation*}
G_{t}(x,y)=f(y)-f(x)-(y-x)\cdot (p_{t}\ast {\nabla }f)(x).
\end{equation*}%
Then there exists $C>0$ depending only on $n$, $w$ and the constants in the
Gaussian estimates, such that%
\begin{equation}
\int_{0}^{\infty }\int_{\mathbb{R}^{n}}\left\vert \int_{\mathbb{R}%
^{n}}V_{t}(x,y)G_{t}(x,y)\,dy\right\vert ^{2}\,w(x)\,dx\,\frac{dt}{t}\leq
C\Vert {\nabla }f\Vert _{L^{2}(w)}^{2}.  \label{eqn:ugly-sqf}
\end{equation}
\end{lemma}

The second square function inequality is needed in the last step of the
proof in Section~\ref{section:test-func}. In the unweighted case this
inequality is due to Journ\'e~\cite{journe91}; our proof is adapted from
that of Auscher and Tchamitchian~\cite{auscher-tchamitchian98} as explicated
by Grafakos~\cite{grafakos04}.

Define the averaging operator $A_{t}$ by 
\begin{equation*}
A_{t}f(x)=%
\mathchoice {{\setbox0=\hbox{$\displaystyle{\textstyle
-}{\int}$} \vcenter{\hbox{$\textstyle -$}}\kern-.5\wd0}} {{\setbox0=\hbox{$\textstyle{\scriptstyle -}{\int}$} \vcenter{\hbox{$\scriptstyle -$}}\kern-.5\wd0}} {{\setbox0=\hbox{$\scriptstyle{\scriptscriptstyle -}{\int}$} \vcenter{\hbox{$\scriptscriptstyle -$}}\kern-.5\wd0}} {{\setbox0=\hbox{$\scriptscriptstyle{\scriptscriptstyle -}{\int}$} \vcenter{\hbox{$\scriptscriptstyle -$}}\kern-.5\wd0}}%
\!\int_{Q_{t}(x)}f(y)\,dy,
\end{equation*}%
where $Q_{t}(x)$ is the unique dyadic cube containing $x$ such that $t\leq
\ell (Q_{t}(x))<2t$. If $2^{-k-1}<t\leq 2^{-k}$, and $Q_{j}^{k}$ denotes the
collection of dyadic cubes of side length $2^{-k}$, then the kernel of $%
A_{t} $ is 
\begin{equation*}
A_{t}(x,y)=\sum_{j}|Q_{j}^{k}|^{-1}\chi _{Q_{j}^{k}}(x)\chi _{Q_{j}^{k}}(y).
\end{equation*}%
It follows immediately from the definition of $A_{t}$ that $|A_{t}f(x)|\leq
Mf(x)$, and so by Lemma~\ref{prop:conv}, $A_{t}$ is bounded on $L^{2}(w)$
for all $w\in A_{2}$ with a constant independent of $t$.

Define the operator $P_tf(x) = p_t*f(x)$, where $p_t(x)=t^{-n}p(x/t)$ and $p$
is a non-negative, radial function such that ${\mathrm{supp}}(p)\subset
B_1(0)$ and $\|p\|_1=1$. Then again by Lemma~\ref{prop:conv}, $|P_tf(x)|\leq
C(n)Mf(x)$ and $P_t$ is uniformly bounded on $L^2(w)$, $w\in A_2$.

\begin{lemma}
\label{lemma:dyadic-conv} Given $w\in A_2$, there exists a constant $C$ such
that for all $f\in L^2(w)$, 
\begin{equation}  \label{eqn:avg1}
\int_0^\infty\int_{{\mathbb{R}^n}} |(P_t-A_t)f(x)|^2 w(x)\,dx\frac{dt}{t}
\leq C\int_{{\mathbb{R}^n}} |f(x)|^2w(x)\,dx.
\end{equation}
\end{lemma}

\begin{remark}
\label{remark:dyadic-conv} Though Lemma~\ref{lemma:dyadic-conv} is stated in
terms of the dyadic grid, it will be clear from the proof that it is true if
we replace the dyadic grid by the ``dyadic'' grid relative to a fixed cube $%
Q $: that is, the collection of cubes gotten as in the construction of the
standard dyadic grid, but starting with $Q$ instead of $[0,1)^n$.
\end{remark}

\subsection*{Proof of Lemma~\protect\ref{lemma-sq-Taylor}}

The proof requires several steps. First, we will show that there exists a
family of sublinear operators $\{R_{t}^{k}\}$, $k\geq 0$, such that %
\eqref{eqn:ugly-sqf} holds provided that there exists $A>1$ such that 
\begin{equation}
\int_{0}^{\infty }\int_{\mathbb{R}^{n}}|R_{t}^{k}F(x)|^{2}\,w(x)\,dx\,\frac{%
dt}{t}\leq CA^{k}\Vert F\Vert _{L^{2}(w)}^{2},  \label{eqn:Rsquare}
\end{equation}%
where 
\begin{equation*}
F(x)=R\cdot {\nabla }f(x)=\sum_{j=1}^{n}R_{j}\left( \frac{\partial f}{%
\partial x_{j}}\right) (x)
\end{equation*}%
and the $R_{j}$ are the Riesz transforms. This square function estimate will
follow from Proposition~\ref{prop:schur} if we can prove that the operators $%
R_{t}^{k}$ are uniformly bounded on $L^{2}(w)$ and satisfy two operator norm
estimates. Let $\psi $ be a radial Schwartz function such that $\widehat{%
\psi }(0)=0$ and such that \eqref{eqn:repro1} holds, and let $Q_{s}f=\psi
_{s}\ast f$. We will first show that for all $s,\,t>0$, 
\begin{equation}
\Vert R_{t}^{k}Q_{s}\Vert _{\mathbf{B}(L^{2}(w))}\leq C(n,[w]_{A_{2}},\psi )
\label{eqn:Restimate1}
\end{equation}%
and then show the stronger estimate 
\begin{equation}
\Vert R_{t}^{k}Q_{s}\Vert _{\mathbf{B}(L^{2}(w))}\leq C(n,[w]_{A_{2}},p,\psi
)A^{k}\min \left( \frac{t}{s},\frac{s}{t}\right) ^{\alpha }.
\label{eqn:Restimate2}
\end{equation}

\medskip

\subsection*{Reduction to \eqref{eqn:Rsquare}}

By Lemma \ref{lemma:gaussian-prime}, 
\begin{align*}
& \left| \int_{\mathbb{R}^{n}}V_{t}(x,y)G_{t}(x,y)\,dy\right| \\
& \qquad \leq C_{1}t^{-n-1}\int_{\mathbb{R}^{n}}\exp \left( -C_{2}\frac{%
|x-y|^{2}}{t^{2}}\right) |G_{t}(x,y)|\,dy \\
& \qquad =C_{1}t^{-n-1}\int_{|x-y|<t}\exp \left( -C_{2}\frac{|x-y|^{2}}{t^{2}%
}\right) |G_{t}(x,y)|\,dy \\
& \qquad \qquad +C_{1}t^{-n-1}\sum_{k=1}^{\infty }\int_{2^{k-1}t\leq
|x-y|<2^{k}t}\exp \left( -C_{2}\frac{|x-y|^{2}}{t^{2}}\right)
|G_{t}(x,y)|\,dy.
\end{align*}

In the first integral, make the change of variables $h=(y-x)/t$; in the
integrals in the sum, make the change of variables $h=(y-x)/2^{k}t$. Then
there exist positive constants $B_1$ and $B_2$ such that 
\begin{equation*}
\left|\int_{\mathbb{R}^n} V_t(x,y)G_t(x,y)\,dy\right| \leq
B_1t^{-1}\sum_{k=0}^\infty 2^{nk}\exp(-B_24^k)\int_{|h|<1}
|G_t(x,x+2^{k}ht)|\,dh.
\end{equation*}

Note that for any $B_{2}>0$, 
\begin{equation*}
\sum_{k=0}^{\infty }2^{nk}\exp (-B_{2}4^{k})<\infty .
\end{equation*}%
Therefore, by H\"{o}lder's inequality we get the following estimate: 
\begin{align}
& \int_{0}^{\infty }\int_{\mathbb{R}^{n}}\left| \int_{\mathbb{R}^n}
V_{t}(x,y)G_{t}(x,y)\,dy\right|^{2}\,w(x)\,dx\,\frac{dt}{t}  \notag \\
& \qquad \leq C\int_{0}^{\infty }\int_{\mathbb{R}^{n}}\left(
t^{-1}\sum_{k=0}^{\infty }2^{nk}\exp
(-B_{2}4^{k})\int_{|h|<1}|G_{t}(x,x+2^{k}ht)|\,dh\right) ^{2}\,w(x)\,dx\,%
\frac{dt}{t}  \notag \\
& \qquad \leq C\sum_{k=0}^{\infty }2^{nk}\exp (-B_{2}4^{k})\int_{0}^{\infty
}\int_{\mathbb{R}^{n}}\left(
t^{-1}\int_{|h|<1}|G_{t}(x,x+2^{k}ht)|\,dh\right) ^{2}\,w(x)\,dx\,\frac{dt}{t%
};  \notag \\
\intertext{if we make the change of variables $t\mapsto 2^{-k}t=t_{k}$, we
get}
& \qquad \leq C\sum_{k=0}^{\infty }2^{(n+2)k}\exp (-B_{2}4^{k})  \notag \\
& \qquad \qquad \qquad \times \int_{0}^{\infty }\int_{\mathbb{R}^{n}}\left(
t^{-1}\int_{|h|<1}|G_{t_{k}}(x,x+ht)|\,dh\right) ^{2}\,w(x)\,dx\,\frac{dt}{t}%
.  \label{eqn:Rprev}
\end{align}

Assume for the moment that for each $k$ there exists a family of sublinear
operators $\{R_{t}^{k}\}$ such that 
\begin{equation}
R_{t}^{k}F(x)=t^{-1}\int_{|h|<1}|G_{t_{k}}(x,x+ht)|\,dh,  \label{eqn:Rdefn}
\end{equation}
where $F=R\cdot {{\nabla} }f$, and that there exists $A>1$ such that %
\eqref{eqn:Rsquare} holds. By Lemma~\ref{prop:sio}, $\Vert F\Vert
_{L^{2}(w)}\leq C(n,[w]_{A_{2}})\Vert {{\nabla} }f\Vert _{L^{2}(w)}$, and so
if we combine \eqref{eqn:Rprev} and \eqref{eqn:Rsquare} we get inequality~(%
\ref{eqn:ugly-sqf}).

\medskip

To construct the operators $R_t^k$ and show that \eqref{eqn:Rdefn} holds,
recall that the the Riesz potential $I_1$ is the convolution operator with
kernel $i_1(x)=c_n|x|^{1-n}$, where the constant $c_n$ is chosen so that 
\begin{equation*}
\widehat{I_1f}(\xi) = (2\pi|\xi|)^{-1}\widehat{f}(\xi).
\end{equation*}
The following identities are well-known (see, for instance, Stein~\cite%
{stein70}): if $f$ is a Schwartz function and $F= R\cdot {\nabla} f$, then $%
f = I_1F$. Further, given any $h\in {\mathbb{R}^n}$, $h\cdot{\nabla} f = -
R\cdot(hF)$.

Define the convolution kernel 
\begin{equation*}
J_{t}^{h}(x)=t^{-1}\left( i_{1}(x+ht)-i_{1}(x)\right) ;
\end{equation*}%
then we have that 
\begin{align*}
t^{-1}G_{t_{k}}\left( x,x+ht\right) & = \frac{f\left( x+ht\right) -f\left(
x\right) }{t}-p_{t_{k}}\ast (h\cdot {{\nabla} }f)\left( x\right) \\
& =\left( J_{t}^{h}\ast F\right) \left( x\right) +p_{t_{k}}\ast (R\cdot
(hF))(x)=H_{t,h}^{k}F(x).
\end{align*}
Therefore, if we define 
\begin{equation*}
R_{t}^{k}F(x)=\int_{|h|<1}|H_{t,h}^{k}F(x)|\,dh,
\end{equation*}
then \eqref{eqn:Rdefn} holds. And, since the operators $H_{t,h}^{k}$ are
linear, each $R_{t}^{k}$ is sublinear.

\subsection*{Proof of inequality \eqref{eqn:Restimate1}}

Since by Lemma~\ref{prop:conv} the operators $Q_s$ are uniformly bounded on $%
L^2(w)$, to prove \eqref{eqn:Restimate1} it will suffice to prove that for
all $t$ and $k$, 
\begin{equation}  \label{eqn:Restimate3}
\|R_t^k\|_{\mathbf{B}(L^2(w))} \leq C(n,[w]_{A_2},\psi).
\end{equation}
By definition, for all Schwartz functions $g$ we have that 
\begin{equation*}
|R_t^kg(x)| \leq \int_{|h|<1} |(J_t^h*g)(x)|\,dh + \int_{|h|<1}
|p_{t_k}*(R\cdot(hg))(x)|\,dh.
\end{equation*}
To prove \eqref{eqn:Restimate3} we will prove that each term on the
right-hand side is uniformly bounded on $L^2(w)$. The boundedness of the
second follows immediately from Lemmas~\ref{prop:conv} and \ref{prop:sio}:
since $\|p\|_1=1$, for all $k$ and $t$ we have that

\begin{align*}
& \int_{\mathbb{R}^{n}}\left( \int_{|h|<1}|p_{t_{k}}\ast (R\cdot
(hg))(x)|\,dh \right)^{2}w(x)\,dx \\
& \qquad \qquad \leq C(n)\int_{|h|<1}\int_{\mathbb{R}^{n}}|p_{t_{k}}\ast
(R\cdot (hg))(x)|^{2}w(x)\,dx\,dh \\
& \qquad \qquad \leq C(n,[w]_{A_{2}})\int_{|h|<1}\int_{\mathbb{R}%
^{n}}|R\cdot (hg)(x)|^{2}w(x)\,dx\,dh \\
& \qquad \qquad \leq C(n,[w]_{A_{2}})\int_{|h|<1}\int_{\mathbb{R}%
^{n}}|hg(x)|^{2}w(x)\,dx\,dh \\
& \qquad \qquad \leq C(n,[w]_{A_{2}})\int_{\mathbb{R}^{n}}|g(x)|^{2}w(x)\,dx.
\end{align*}

We will now prove that the first term is bounded on $L^2(w)$: 
\begin{equation}  \label{eqn:firstterm}
\int_{\mathbb{R}^n} \left(\int_{|h|<1} |J_t^h*g(x)|\,dh\right)^2 w(x)\,dx
\leq C(n,[w]_{A_2})\|g\|_{L^2(w)}^2.
\end{equation}

To prove this we first estimate the inner integral on the left-hand side: 
\begin{align*}
\int_{|h|<1}|J_{t}^{h}\ast g(x)|\,dh & \leq \int_{|h|<1} \bigg| \left(
J_{t}^{h}+c_{n}(n-1)\frac{h\cdot x}{|x|^{n+1}}\chi _{\{|x/t|>2\}}\right)
\ast g(x)\bigg| \,dh \\
& \qquad + \int_{|h|<1}\bigg| c_{n}(n-1)\frac{h\cdot x}{|x|^{n+1}}\chi
_{\{|x/t|>2\}}\ast g(x) \bigg| \,dh \\
& \leq \bigg( \int_{|h|<1}\bigg| J_{t}^{h}+c_{n}(n-1)\frac{h\cdot x}{%
|x|^{n+1}}\chi _{\{|x/t|>2\}}\bigg| \,dh\bigg) \ast |g|(x) \\
& \qquad +C_{n}\sum_{i=1}^{n}\bigg| \frac{x_{i}}{|x|^{n+1}}\chi
_{\{|x/t|>2\}}\ast g(x)\bigg| \\
& \leq (L_{t}\ast |g|)(x)+C(n)\sum_{i=1}^{n}R_{i}^{\ast }g(x),
\end{align*}
where 
\begin{equation*}
L(x)=\int_{|h|<1}\bigg| J_{1}^{h}(x)+c_{n}(n-1)\frac{h\cdot x}{|x|^{n+1}}%
\chi _{|x|>2}(x)\bigg| \,dh,
\end{equation*}
$c_{n}$ is the constant in the Riesz potential and $R_{i}^{\ast }$ is the
maximal singular integral associated with the Riesz transform $R_{i}$.

By Lemma~\ref{prop:sio}, 
\begin{equation*}
\int_{\mathbb{R}^{n}}| \sum_{i=1}^{n}R_{i}^{\ast }g(x)| ^{2}w(x)\,dx\leq
C(n,[w]_{A_{2}})\Vert g\Vert _{L^{2}(w)}^{2}.
\end{equation*}
Therefore, to complete the proof of \eqref{eqn:firstterm} we need to show
that 
\begin{equation*}
\int_{\mathbb{R}^{n}}|(L_{t}\ast | g| )(x)|^{2}w(x)\,dx\leq
C(n,[w]_{A_{2}})\Vert g\Vert _{L^{2}(w)}^{2}.
\end{equation*}
But by Lemma~\ref{prop:conv} it will suffice to prove that 
\begin{equation}
|L(x)|\leq C(n)\min \left( \frac{1}{|x|^{n-1}},\frac{1}{|x|^{n+1}}\right) ,
\label{eqn:rad-bdd}
\end{equation}
since the right-hand side is a radial, decreasing function in $L^{1}$. To
prove this estimate we treat several cases depending on the size of $x$.

\textsc{Case 1: $|x|>2$}. In this case, since $h\cdot {{\nabla} }%
i_{1}(x)=c_{n}(n-1)h\cdot x/|x|^{n+1}$, if we apply the mean value theorem
twice we get that 
\begin{align*}
|L(x)|& =c_{n}\int_{|h|<1}| | x+h| ^{1-n}-| x| ^{1-n}+(n-1)\frac{h\cdot x}{%
|x|^{n+1}}| \,dh \\
& \leq C(n)\int_{|h|<1}\frac{1}{|x|^{n+1}}\,dh\leq \frac{C(n)}{|x|^{n+1}}.
\end{align*}

\textsc{Case 2: $|x|\leq 2$}. In this case we have that 
\begin{equation}
|L(x)|\leq \sum_{k=0}^{\infty }\int_{\{2^{-k-1}\leq
|h|<2^{-k}\}}|i_{1}(x+h)-i_{i}(x)|\,dh.  \label{eqn:case2}
\end{equation}%
Fix $k$. Then we consider the following sub-cases.

\textsc{Sub-case 2.1: $|x|>2^{-k+1}$.} In this case, $|x|/2>|h|$. Hence, $%
|x+h|>|x|/2$, so 
\begin{equation}  \label{eqn:kbound}
\int_{\{2^{-k-1}\leq |h|<2^{-k}\}} |i_1(x+h)-i_i(x)|\,dh \leq C(n)\frac{%
2^{-kn}}{|x|^{n-1}}.
\end{equation}

\textsc{Sub-case 2.2: $|x|<2^{-k-2}$}. In this case, we have that $%
|x+h|>|h|/2>2^{-k-2}$, so we can argue as in the previous case to get %
\eqref{eqn:kbound}.

\textsc{Sub-case 2.3: $2^{-k-2}\leq |x|\leq 2^{-k+1}$}. In this case we
estimate as follows :%
\begin{align*}
\int_{\{2^{-k-1}\leq |h|<2^{-k}\}}|i_{1}(x+h)-i_{i}(x)|\,dh & =
c_{n}\int_{\{2^{-k-1}\leq |h|<2^{-k}\}}\big|| x+h| ^{1-n}-| x| ^{1-n}\big|%
\,dh \\
& = \frac{c_{n}}{| x| ^{n-1}}\int_{\{2^{-k-1}\leq |h|<2^{-k}\}}\big|| \frac{x%
}{| x| }+\frac{h}{| x| }| ^{1-n}-1\big|\,dh \\
\intertext{if we make the change of variables $u=h/|x|$, we get}
& = \frac{c_{n}| x| ^{n}}{| x| ^{n-1}}\int_{\{\frac{2^{-k-1}}{| x| }\leq |u|<%
\frac{2^{-k}}{| x| }\}}\big|| \frac{x}{| x| }+u| ^{1-n}-1\big|\,du \\
& \leq \frac{C_{n}2^{-kn}}{|x|^{n-1}}\int_{\{\frac{1}{4}\leq |u|<4\}}\big|| 
\frac{x}{| x| }+u| ^{1-n}-1\big|\,du \\
& = \frac{C_{n}2^{-kn}}{|x|^{n-1}}\int_{\{\frac{1}{4}\leq |u|<4\}}\big||
1+u| ^{1-n}-1\big|\,du,
\end{align*}
where the last equality holds by rotational symmetry. Note that this last
integral is finite and its value depends on $n$ but is independent of $k$
and $x$.

If we apply these three sub-cases to \eqref{eqn:case2}, we get that for all $%
|x|\leq 2$, 
\begin{equation*}
|L(x)| \leq \sum_{k=0}^\infty \frac{C(n)2^{-kn}}{|x|^{n-1}} \leq \frac{C(n)}{%
|x|^{n-1}}.
\end{equation*}

This completes our proof of (\ref{eqn:rad-bdd}), and so of %
\eqref{eqn:firstterm}. Therefore, we have shown that inequality %
\eqref{eqn:Restimate3} holds.

\subsection*{Proof of Inequality \eqref{eqn:Restimate2}}

By Lemma~\ref{lemma:interpol} and inequality \eqref{eqn:Restimate1}, to
prove \eqref{eqn:Restimate2} it will suffice to show the corresponding
unweighted norm estimate: 
\begin{equation}  \label{eqn:Restimate4}
\|R_t^kQ_s\|_{\mathbf{B}(L^2)} \leq C(n,[w]_{A_2},p,\psi)A^k\min\left(\frac{t%
}{s},\frac{s}{t}\right).
\end{equation}

Fix $s,\,t>0$, $k\geq 0$, and recall that $Q_sf= \psi_s*f$. By H\"older's
inequality and Plancherel's theorem, 
\begin{align*}
& \int_{\mathbb{R}^n} |R_t^kQ_sg(x)|^2\,dx \\
& \qquad \leq C(n) \int_{|h|<1}\int_{\mathbb{R}^n}
|J_t^h*\psi_s*g(x)+p_{t_k}*(R\cdot(h(\psi_s*g))(x)|^2\,dx\,dh \\
& \qquad = C(n) \int_{|h|<1}\int_{\mathbb{R}^n} \left|\widehat{J_t^h}(\xi)%
\widehat{\psi}(s\xi)\widehat{g}(\xi) -i \,\widehat{p}(t_k\xi) \frac{\xi\cdot
h}{|\xi|}\widehat{\psi}(s\xi)\widehat{g}(\xi) \right|^2\,d\xi\,dh \\
& \qquad = C(n) \int_{|h|<1}\int_{\mathbb{R}^n} \left|\frac{e^{2\pi i
th\cdot\xi}-1- 2\pi i \widehat{p}(t_k\xi)\xi\cdot h}{2\pi t|\xi|}\right|^2 |%
\widehat{\psi}(s\xi)|^2|\widehat{g}(\xi)|^2\,d\xi\,dh
\end{align*}

To estimate the last term, we will use the following: since $\psi $ is a
Schwartz function such that $\widehat{\psi }(0)=0$, 
\begin{equation*}
|\widehat{\psi }(\xi )|\leq C(\psi )\min (|s\xi |,|s\xi |^{-1}).
\end{equation*}%
Since $p\in C_{c}^{\infty }$, 
\begin{equation*}
|\widehat{p}(t_{k}\xi )|\leq \frac{C(p)}{t_{k}|\xi |}=\frac{C(p)2^{k}}{t|\xi
|}.
\end{equation*}%
And, since $\widehat{p}(0)=1$, 
\begin{equation*}
e^{2\pi ith\cdot \xi }-1-2\pi i\,\widehat{p}(t_{k}\xi )\xi \cdot h=O(|t\xi
|^{2}).
\end{equation*}%
Therefore, we have that 
\begin{equation*}
\left\vert \frac{e^{2\pi ith\cdot \xi }-1-2\pi i\widehat{p}(t_{k}\xi )\xi
\cdot h}{2\pi t|\xi |}\right\vert \leq C(p)2^{k}\min (|t\xi |,|t\xi |^{-1}).
\end{equation*}%
Combining these estimates we get that 
\begin{align*}
& \int_{\mathbb{R}^{n}}|R_{t}^{k}Q_{s}g(x)|^{2}\,dx \\
& \qquad \leq C(n,p,\psi )4^{k}\int_{|h|<1}\int_{\mathbb{R}^{n}}\min (|t\xi
|,|t\xi |^{-1})^{2}\min (|s\xi |,|s\xi |^{-1})^{2}|\widehat{g}(\xi
)|^{2}\,d\xi \,dh \\
& \qquad \leq C(n,p,\psi )4^{k}\min \left( \frac{s}{t},\frac{t}{s}\right)
^{2}\int_{|h|<1}\int_{\mathbb{R}^{n}}|\widehat{g}(\xi )|^{2}\,d\xi \,dh \\
& \qquad \leq C(n,p,\psi )4^{k}\min \left( \frac{s}{t},\frac{t}{s}\right)
^{2}\int_{\mathbb{R}^{n}}|g(x)|^{2}\,dx.
\end{align*}%
Inequality \eqref{eqn:Restimate4} follows immediately. This finishes the
proof of Lemma \ref{lemma-sq-Taylor}.

\subsection*{Proof of Lemma~\protect\ref{lemma:dyadic-conv}}

For brevity, let $R_t=P_t-A_t$. Then we have that 
\begin{equation*}
R_tf = R_t(I-P_t)f + R_tP_tf,
\end{equation*}
so it will suffice to prove that \eqref{eqn:avg1} holds with $R_t$ replaced
by $R_t(I-P_t)$ and $R_tP_t$. By our remarks above, it is clear that both of
these operators are uniformly bounded on $L^2(w)$. Therefore, by Lemma~\ref%
{lemma:interpol} and Proposition~\ref{prop:schur} it will suffice to show
that for all $s,\,t>0$, there exist constants $C,\,\alpha >0$ such that 
\begin{gather}
\|R_t(I-P_t)Q_s\|_{L^2} \leq C\min\left(\frac{t}{s},\frac{s}{t}%
\right)^\alpha,  \label{eqn:bnd1} \\
\|R_tP_tQ_s\|_{L^2} \leq C\min\left(\frac{t}{s},\frac{s}{t}\right)^\alpha,
\label{eqn:bnd2}
\end{gather}
where $Q_sf= \psi_s*f(x)$ with $\psi$ a radial Schwartz function with $%
\widehat{\psi}(0) = 0$.

We first prove \eqref{eqn:bnd1}. Since $R_t$ is uniformly bounded on $L^2$, 
\begin{equation*}
\|R_t(I-P_t)Q_s\|_{L^2} \leq C\|(I-P_t)Q_s\|_{L^2}.
\end{equation*}
We can bound the right-hand side using Plancherel's theorem. By our choice
of $\psi$ and since $\hat{p}(0)=1$, 
\begin{equation*}
|\hat{\psi}(s\xi)| \leq \frac{C(\psi)}{|s\xi|}, \qquad |1-\hat{p}(t\xi)|
\leq C|t\xi|.
\end{equation*}
Therefore, 
\begin{multline*}
\int_{{\mathbb{R}^n}} |(I-P_t)Q_sf(x)|^2\,dx = \int_{{\mathbb{R}^n}} |(1-%
\hat{p}(t\xi)|^2|\hat{\psi}(s\xi)^2||\hat{f}(\xi)|^2\,d\xi \\
\leq C \left(\frac{t}{s}\right)^2\int_{{\mathbb{R}^n}} |\hat{f}%
(\xi)|^2\,d\xi = C \left(\frac{t}{s}\right)^2\|f\|_2^2,
\end{multline*}
and so 
\begin{equation*}
\|R_t(I-P_t)Q_s\|_{L^2} \leq C\left(\frac{t}{s}\right).
\end{equation*}

On the other hand, since convolution operators commute and $I-P_t$ is
uniformly bounded on $L^2$, 
\begin{equation*}
\|R_t(I-P_t)Q_s\|_{L^2} = \|R_tQ_s(I-P_t)\|_{L^2} \leq C\|R_tQ_s\|_{L^2}.
\end{equation*}
For any $\alpha$, $0<\alpha<1/2$, there exists a constant $C$ such that 
\begin{equation*}
\|A_tQ_s \|_{L^2} \leq C\left(\frac{s}{t}\right)^\alpha.
\end{equation*}
(See Grafakos~\cite{grafakos04}.) Further, since 
\begin{equation*}
|\hat{p}(t\xi)| \leq \frac{C}{|t\xi|}, \qquad |\hat{\psi}(s\xi)| \leq
C|s\xi|,
\end{equation*}
we can again use Plancherel's theorem to see that 
\begin{equation*}
\|P_tQ_s\|_{L^2} \leq C\left(\frac{s}{t}\right).
\end{equation*}
It follows that for $\alpha<1/2$ 
\begin{equation*}
\|R_tQ_s\|_{L^2} \leq C\left(\frac{s}{t}\right)^{\alpha},
\end{equation*}
and so we get that \eqref{eqn:bnd1} holds.

\medskip

To prove \eqref{eqn:bnd2} we will apply Lemma~\ref{lemma:reg-family}. It
will suffice to show that $R_tP_t(1)=0$, and that the kernel $K_t$ of the
operator $R_tP_t$ satisfies \eqref{eqn:grafakos1} and \eqref{eqn:grafakos2}.
The identity is immediate: both $A_t$ and $P_t$ are bounded on $L^\infty$
and $A_t(1)=P_t(1)=1$.

Let $K_t = J_t + L_t$, where $J_t$ is the kernel of $P_t^2$ and $L_t$ is the
kernel of $A_tP_t$, that is, 
\begin{gather*}
J_t(x,y) = \int_{{\mathbb{R}^n}} p_t(x-z)p_t(y-z)\,dz, \\
L_t(x,y) = \int_{\mathbb{R}^n} \sum_j
|Q_j^k|\chi_{Q_j^k}(x)\chi_{Q_j^k}(z)p_t(y-z)\,dz.
\end{gather*}
It is immediate from these expressions that there exists a constant $c>0$
such that both $J_t(x,y)$ and $L_t(x,y)$ are non-zero only if $|x-y|<ct$.
Further, we have that 
\begin{equation*}
|J_t(x,y)| \leq t^{-n}\|p\|_{L^\infty}\|p\|_{L^1} \leq Ct^{-n}.
\end{equation*}
Similarly, since $|Q_j^k|\approx t^{-n}$, 
\begin{equation*}
|L_t(x,y)| \leq Ct^{-n}\|p\|_{L^\infty}.
\end{equation*}
Inequality \eqref{eqn:grafakos1} follows at once.

The proof of \eqref{eqn:grafakos2} is similar. By the mean value theorem, 
\begin{equation*}
|p_{t}(y-z)-p_{t}(y^{\prime}-z)|\leq t^{-n}\Vert {\nabla }p\Vert _{L^{\infty
}}\left\vert \frac{y-y^{\prime }}{t}\right\vert =\frac{C|y-y^{\prime }|}{%
t^{n+1}}.
\end{equation*}%
If we use this to estimate $|J_{t}(x,y)-J_{t}(x,y^{\prime })|$ and $%
|L_{t}(x,y)-L_{t}(x,y^{\prime })|$ and argue as before, we get that both
satisfy a similar bound. Inequality \eqref{eqn:grafakos2} follows at once.
Therefore, we have proved \eqref{eqn:bnd2} and our proof is complete.

\section{Reduction to a Square Function Estimate\label%
{section:square-functions}}

In this section we begin our proof of Theorem~\ref{theorem:WKato} by proving
that inequality~(\ref{eqn:kato-half}) holds if we have the square function
estimate 
\begin{equation}  \label{eqn:firstreduce}
\left( \int_{0}^{\infty }\int_{\mathbb{R}^{n}}| V_{t}f(x)| ^{2}\,w(x)\,dx\,%
\frac{dt}{t}\right) ^{1/2}\leq C\Vert {{\nabla} }f\Vert _{L^{2}(w)},
\end{equation}
where $V_{t}=-2t{\mathcal{L}}_{w}e^{-t^{2}{\mathcal{L}}_{w}}$. Fix $f\in D({%
\mathcal{L}}_{w})$; recall that $D({\mathcal{L}}_{w})$ is dense in $%
H_{0}^{1}(w)$ (see \cite{cruz-uribe-riosP}). If we apply integration by
parts to a well known formula (see, for instance, Kato \cite{kato66}) we get
that 
\begin{equation}
{\mathcal{L}}_{w}^{1/2}f=\int_{0}^{\infty }t^{3}e^{-2t^{2}{\mathcal{L}}_{w}}{%
\mathcal{L}}_{w}^{2}f\,\frac{dt}{t}.  \label{eqn:sqr-func-calc}
\end{equation}
Therefore, for all $g\in L^{2}(w)$, 
\begin{align*}
| \langle {\mathcal{L}}_{w}^{1/2}f,g\rangle _{w}| ^{2}& = \left|
\int_{0}^{\infty }\left\langle t^{3}e^{-2t^{2}{\mathcal{L}}_{w}}{\mathcal{L}}%
_{w}^{2}f,g\right\rangle _{w} \frac{dt}{t}\right| ^{2} \\
& = \left| \int_{0}^{\infty }\left\langle te^{-t^{2}{\mathcal{L}}_{w}}{%
\mathcal{L}}_{w}f,t^{2}{\mathcal{L}}_{w}^{\ast }e^{-t^{2}{\mathcal{L}}%
_{w}^{\ast }}g\right\rangle _{w}\frac{dt}{t}\right| ^{2} \\
& \leq \left( \int_{0}^{\infty }\int_{\mathbb{R}^{n}}|te^{-t^{2}{\mathcal{L}}%
_{w}}{\mathcal{L}}_{w}f(x)|^{2}\,w(x)\,\frac{dt}{t}\right) \\
& \qquad \times \left( \int_{0}^{\infty }\int_{\mathbb{R}^{n}}|t^{2}{%
\mathcal{L}}_{w}^{\ast }e^{-t^{2}{\mathcal{L}}_{w}^{\ast }}g(x)|^{2}\,w(x)\,%
\frac{dt}{t}\right) .
\end{align*}

Hence, by duality we have shown that \eqref{eqn:kato-half} follows from~(\ref%
{eqn:firstreduce}) provided that we can prove the square function inequality 
\begin{equation*}
\int_{0}^{\infty }\int_{\mathbb{R}^n}| t^{2}{\mathcal{L}}_w^* e^{-t^{2}{%
\mathcal{L}}_w^{*}}g(x)|^{2}\,w(x)\,\frac{dt}{t} \leq C\|g\|^2_{L^2(w)}.
\end{equation*}

Since the semigroups $e^{-t{\mathcal{L}}_{w}^{\ast }}$ and $e^{-t{\mathcal{L}%
}_{w}}$ satisfy the same estimates, this is equivalent to proving that 
\begin{equation}
\int_{0}^{\infty }\int_{\mathbb{R}^{n}}|tV_{t}g(x)|^{2}\,w(x)\,\frac{dt}{t}%
\leq C\Vert g\Vert _{L^{2}(w)}^2.  \label{eqn:reduce2}
\end{equation}

To prove this square function estimate we use Proposition~\ref{prop:schur}.
Let $G(x)=\exp (-C_{2}|x|^{2})$. Then by Lemma \ref{lemma:gaussian-prime}, 
\begin{eqnarray*}
|tV_{t}f(x)| &\leq &\int_{\mathbb{R}^{n}}|tV_{t}(x,y)||f(y)|\,dy \\
&\leq &C_{1}\int_{\mathbb{R}^{n}}t^{-n}\exp \left( -C_{2}\frac{|x-y|^{2}}{%
t^{2}}\right) |f(y)|\,dy=C_{1}(G_{t}\ast |f|)(x),
\end{eqnarray*}%
Since $G\in L^{1}$ and is radial, by Lemma~\ref{prop:conv}, $%
\sup_{t>0}(G_{t}\ast |f|)(x)\leq CMf(x)$ and the operators $tV_{t}$ are
uniformly bounded on $L^{2}(w)$. Let $\psi $ be a radial Schwartz function
such that $\widehat{\psi }(0)=0$. For $s>0$, let $Q_{s}=\psi _{s}\ast f$.
Again by Lemma~\ref{prop:conv} we have that the operators $Q_{s}$ are
uniformly bounded on $L^{2}(w)$. Therefore, there exists a constant $C$ such
that for all $s,\,t>0$, 
\begin{equation*}
\Vert tV_{t}Q_{s}\Vert _{{\mathcal{B}}(L^{2}(w))}\leq C.
\end{equation*}

Further, by Lemmas~\ref{lemma:gaussian-prime} and~\ref{lemma:reg-family}
there exists $\beta >0$ such that 
\begin{equation*}
\Vert tV_{t}Q_{s}\Vert _{{\mathcal{B}}(L^{2})}\leq C\min \left( \frac{t}{s},%
\frac{s}{t}\right) ^{\beta }.
\end{equation*}%
Hence, by Lemma~\ref{lemma:interpol} we have that for some $\alpha >0$, 
\begin{equation*}
\Vert tV_{t}Q_{s}\Vert _{{\mathcal{B}}(L^{2}(w))}\leq C\min \left( \frac{t}{s%
},\frac{s}{t}\right) ^{\alpha }.
\end{equation*}%
Therefore, the operators $\{tV_{t}\}$ satisfy the hypotheses of Lemma~\ref%
{prop:schur}, so \eqref{eqn:reduce2} holds.

\section{Reduction to a Carleson Measure Estimate\label{section:Carleson}}

In this section we prove that \eqref{eqn:firstreduce} holds provided that we
have a Carleson measure estimate. More precisely, we will show that if $w\in
A_2$, then for all Schwartz functions~$f$, 
\begin{equation}  \label{eqn:goal}
\int_0^\infty \int_{\mathbb{R}^n} |V_tf(x)|^2 \,w(x)\,dx\,\frac{dt}{t} \leq
C(1+\|\gamma_t\|_{C,w})\|{\nabla} f\|_{L^2(w)}^2,
\end{equation}
where $\gamma_t=V_t\phi$, with $\phi(x)=x$.

To prove this, we first show the role played by the Carleson measure
estimate. Let $p\in C_{c}^{\infty }$ be a non-negative, radial, decreasing
function such that $\Vert p\Vert _{1}=1$ and $\mathrm{{\mathrm{supp}}}%
(p)\subset B_{1}(0)$. By Lemma~\ref{lemma:gaussian-prime}, $V_{t}1=0$, so 
\begin{align}
V_{t}f(x)& =\int_{\mathbb{R}^{n}}V_{t}(x,y)f(y)\,dy  \notag \\
& =\left( \int_{\mathbb{R}^{n}}V_{t}(x,y)y\,dy\right) \cdot (p_{t}\ast {%
\nabla }f)(x)  \notag \\
& \qquad +\int_{\mathbb{R}^{n}}V_{t}(x,y)\left[ f(y)-f(x)-(y-x)\cdot
(p_{t}\ast {\nabla }f)(x)\right] \,dy  \notag \\
& =\gamma _{t}(x)\cdot (p_{t}\ast {\nabla }f)(x)+\int_{\mathbb{R}%
^{n}}V_{t}(x,y)G_{t}(x,y)\,dy,  \label{eqn:pieces}
\end{align}%
where $G_{t}(x,y)=f(y)-f(x)-(y-x)\cdot (p_{t}\ast {\nabla }f)(x)$. The first
term in \eqref{eqn:pieces} satisfies a square function estimate. This
follows from the Carleson measure estimate: if $\Vert \gamma _{t}\Vert
_{C,w}<\infty $, then by Lemma \ref{lemma:wtd-carleson}, 
\begin{equation*}
\int_{0}^{\infty }\int_{\mathbb{R}^{n}}|(p_{t}\ast {\nabla }%
f)(x)|^{2}|\gamma _{t}(x)|^{2}\,w(x)\,dx\,\frac{dt}{t}\leq C\Vert \gamma
_{t}\Vert _{C,w}\Vert {\nabla }f\Vert _{L^{2}(w)}^{2}.
\end{equation*}%
The proof of~(\ref{eqn:goal}) now follows from Lemma \ref{lemma-sq-Taylor}.

\section{The Weighted $Tb$ Theorem for Square Roots \label{section:Tb}}

We have reduced the proof of Theorem~\ref{theorem:WKato} to proving that $%
\gamma _{t}\left( x\right) =V_{t}\phi \left( x\right) $ is a Carleson
measure with respect to the weight $w\left( x\right) $: that is, 
\begin{equation}
\| \gamma _{t}\| _{C,w}=\sup_{Q}\frac{1}{w\left( Q\right) }%
\int_{Q}\int_{0}^{\ell \left( Q\right) }| \gamma _{t}\left( x\right) | ^{2}%
\frac{dt}{t}w\left( x\right) dx<\infty,  \label{eqn:Carleson}
\end{equation}
where $\gamma _{t}\left( x\right) =t{\mathcal{L}}_we^{-t^{2}{\mathcal{L}}_w}{%
\varphi }\left( x\right) $ and ${\varphi }\left( x\right) =x$.

In order to prove this fact we will establish a $Tb$ theorem for square
roots, a weighted version of a result due to Auscher and Tchamitchian~\cite%
{auscher-tchamitchian98}. For technical reasons we actually need a slightly
different theorem that is given in Lemma~\ref{lemma-TbRoot-mod} below.
However, it seemed clearer to start with this simpler version and then
sketch the modifications needed to prove the full result.

\begin{lemma}
\label{lemma-TbRoot}Suppose that for every cube $Q$ there exists a mapping $%
F_{Q}:5Q\rightarrow \mathbb{C}^{n}$ and a constant $C=C\left( n,\lambda
,\Lambda ,w\right) <\infty $ such that%
\renewcommand{\theenumi}{\roman{enumi}}%

\begin{enumerate}
\item \renewcommand{\itemsep}{6pt}

\item \label{1:16i}$\displaystyle \int_{5Q}\left\vert \nabla F_{Q}\left(
x\right) \right\vert ^{2}w\left( x\right) dx\leq C w\left( Q\right) $,

\item \label{1:16ii}$\displaystyle \int_{10Q}\left\vert \mathcal{L}%
_{w}F_{Q}\left( x\right) \right\vert ^{2}w\left( x\right) dx\leq Cw\left(
Q\right) \ell \left( Q\right) ^{-2}$,

\item \label{1:16iii} 
\begin{equation*}
\left\Vert \gamma _{t}\right\Vert _{C,w}^{2} \leq C\left\{ \sup_{Q}\frac{1}{%
w\left( Q\right) }\int_{Q}\int_{0}^{\ell \left( Q\right) }\left\vert \gamma
_{t}\left( x\right) \nabla P_{t}F_{Q}\left( x\right) \right\vert ^{2}\frac{dt%
}{t}w\left( x\right) dx+\sup_{t>0}\left\Vert \gamma _{t}\right\Vert _{\infty
}^{2}\right\},
\end{equation*}
\end{enumerate}

where $F_{Q}=\left( F_{1},F_{2},\dots ,F_{n}\right) \in \mathbb{C}^{n}$, and 
$\nabla F_{Q}$ is the matrix $\left( \frac{\partial }{x_{i}}F_{j}\right)
_{i,j=1}^{n}$. Then $\left\Vert \gamma _{t}\right\Vert _{C,w}\leq C\left(
n,\lambda ,\Lambda ,w\right) <\infty $, i.e., $\gamma _{t}$ is a Carleson
measure with respect to $w$.
\end{lemma}

\begin{proof}
We follow the proof of the unweighted lemma in \cite{hofmann01}. Since the
kernel of $V_{t}$ satisfies the Gaussian bounds (\ref{ineq:gauss-p}), we
have that%
\begin{equation}
\sup_{t>0}\left\Vert \gamma _{t}\right\Vert _{\infty }\leq C\left( n,\lambda
,\Lambda ,w\right) <\infty ,  \label{ineq:gammat-b}
\end{equation}%
Indeed, since $V_{t}$ has zero moment, 
\begin{equation*}
\gamma _{t}\left( x\right) =V_{t}\phi \left( x\right) =\int_{\mathbb{R}%
^{n}}V_{t}(x,y)y~dy=-\int_{\mathbb{R}^{n}}V_{t}(x,y)\left( x-y\right) ~dy;
\end{equation*}%
thus, applying the Gaussian estimates, 
\begin{multline}  \label{eqn-sup-gamma}
\left\vert \gamma _{t}\left( x\right) \right\vert \leq \frac{C_{1}}{t^{n+1}}%
\int_{\mathbb{R}^{n}}\exp \left( -C_{2}\frac{|x-y|^{2}}{t^{2}}\right)
\left\vert x-y\right\vert \, dy \\
=C_{1}\int_{0}^{\infty }\int_{S^{n-1}}\exp \left( -C_{2}\frac{r^{2}}{t^{2}}%
\right) \,\left( \dfrac{r}{t}\right) ^{n}d\omega \,\dfrac{dr}{t} \leq
C_{1}\left\vert S^{n-1}\right\vert \int_{0}^{\infty }\exp \left(
-C_{2}s^{2}\right) s^{n}\,ds<\infty ,
\end{multline}
and this bound is independent of $t$. Hence, given (\ref{1:16iii}) to prove (%
\ref{eqn:Carleson}) we only need to show that 
\begin{equation}
\sup_{Q}\frac{1}{w\left( Q\right) }\int_{Q}\int_{0}^{\ell \left( Q\right)
}\left\vert \gamma _{t}\left( x\right) \nabla P_{t}F_{Q}\left( x\right)
\right\vert ^{2}\frac{dt}{t}w\left( x\right) dx\leq C\left( n,\lambda
,\Lambda ,w\right) <\infty .  \label{eqn:tb1}
\end{equation}
Recall that $P_{t}g=p_{t}\ast g$ with $p_{t}\left( x\right) =t^{-n}p\left( 
\frac{x}{t}\right) $, $p\in \mathcal{C}_{0}^{\infty }\left( B_{1}\left(
0\right) \right) $ is radial and $\int p\left( x\right) dx=1$. Let $\tilde{%
\chi}\left( x\right) \in \mathcal{C}_{0}^{\infty }\left( \mathbb{R}%
^{n}\right) $ satisfy 
\begin{equation*}
{\mathrm{supp}}\left( \tilde{\chi}\right) \subset 4Q, \quad 0 \leq \tilde{%
\chi} \leq 1, \quad \left. \tilde{\chi}\right\vert _{3Q}\equiv 1,\quad
\left\vert \nabla \tilde{\chi}\right\vert \leq C\ell \left( Q\right) ^{-1}.
\end{equation*}
Since $p$ is supported in the unit ball and the gradient operator commutes
with $P_{t}$, for $x\in Q$ and $0\leq t\leq \ell \left( Q\right) $ we have
that 
\begin{multline}
\nabla P_{t}F_{Q}\left( x\right) = P_{t}\nabla F_{Q}\left( x\right)
=\int_{B_{t}\left( 0\right) }p_{t}\left( y\right) \left( \nabla \tilde{\chi}%
F_{Q}\right) \left( x-y\right) ~dy \\
= P_{t}\nabla \tilde{\chi}F_{Q}\left( x\right) =P_{t}\nabla \tilde{\chi}%
\left( F_{Q}\left( x\right) -c_{0}\right),  \label{eqn:tb1.5}
\end{multline}
where $c_{0}$ is a constant to be fixed below. Let $\tilde{F}_{Q}=F_{Q}-c_{0}
$; then (\ref{eqn:tb1}) is equivalent to 
\begin{equation}
\sup_{Q}\frac{1}{w\left( Q\right) }\int_{Q}\int_{0}^{\ell \left( Q\right)
}\left\vert \gamma _{t}\left( x\right) P_{t}\nabla \tilde{\chi}\tilde{F}%
_{Q}\left( x\right) \right\vert ^{2}\frac{dt}{t}w\left( x\right) dx\leq
C\left( n,\lambda ,\Lambda ,w\right) <\infty .  \label{eqn:tb2}
\end{equation}
Recall that $\frac{1}{2}V_{t}=-t{\mathcal{L}}_{w}e^{-t^{2}{\mathcal{L}}%
_{w}}=-te^{-t^{2}{\mathcal{L}}_{w}}{\mathcal{L}}_{w}=te^{-t^{2}{\mathcal{L}}%
_{w}}w^{-1}\func{div}\mathbf{A}\nabla $. Define $\theta _{t}=2te^{-t^{2}%
\mathcal{L}_{w}}w^{-1}\func{div}\mathbf{A}$; then $\theta_t$ acts on $%
n\times n$ matrix valued functions. If $\mathbf{1}$ is the $n\times n$
identity matrix, $\gamma _{t}=\theta _{t}\nabla \phi =\theta _{t}\mathbf{1}$%
, so for $x\in Q$ and $0\leq t\leq \ell \left( Q\right) $ we can write 
\begin{eqnarray*}
\gamma _{t}\left( x\right) \nabla P_{t}\tilde{\chi}\tilde{F}_{Q}\left(
x\right) &=&\theta _{t}\mathbf{1}\left( x\right) \cdot P_{t}\left( \nabla 
\tilde{\chi}\tilde{F}_{Q}\right) \left( x\right) \\
&=&\left[ \theta _{t}\mathbf{1}\left( x\right) \cdot P_{t}-\theta _{t}\right]
\left( \nabla \tilde{\chi}\tilde{F}_{Q}\right) \left( x\right) +\theta
_{t}\left( \nabla \tilde{\chi}\tilde{F}_{Q}\right) \left( x\right) \\
&=&R_{t}\left( \nabla \tilde{\chi}\tilde{F}_{Q}\right) \left( x\right)
+\theta _{t}\left( \nabla \tilde{\chi}\tilde{F}_{Q}\right) \left( x\right) ,
\end{eqnarray*}
where $R_{t}=\theta _{t}\mathbf{1}\left( x\right) P_{t}-\theta _{t}$. We
claim that 
\begin{equation}
\frac{1}{w\left( Q\right) }\int_{Q}\int_{0}^{\ell \left( Q\right)
}\left\vert R_{t}\left( \nabla \tilde{\chi}\tilde{F}_{Q}\right) \left(
x\right) \right\vert ^{2}\frac{dt}{t}w\left( x\right) dx\leq \frac{C}{%
w\left( Q\right) }\int_{5Q}\left\vert \nabla F_{Q}\right\vert ^{2}w\left(
x\right) dx.  \label{ineq:sqrf}
\end{equation}
If this is the case, then by assumption \eqref{1:16i} 
\begin{equation*}
\frac{1}{w\left( Q\right) }\int_{Q}\int_{0}^{\ell \left( Q\right)
}\left\vert R_{t}\left( \nabla \tilde{\chi}\tilde{F}_{Q}\right) \left(
x\right) \right\vert ^{2}\frac{dt}{t}w\left( x\right) dx\leq C<\infty .
\end{equation*}
Furthermore, since the operator $e^{-t^{2}\mathcal{L}_{w}}$ is a contraction
in $L^{2}\left( w\right) $, by \eqref{1:16ii}, we have 
\begin{align*}
&\frac{1}{w\left( Q\right) }\int_{Q}\int_{0}^{\ell \left( Q\right)
}\left\vert \theta _{t}\left( \nabla \tilde{\chi}\tilde{F}_{Q}\right) \left(
x\right) \right\vert ^{2}\frac{dt}{t}w\left( x\right) dx \\
& \qquad =\frac{4}{w\left( Q\right) }\int_{Q}\int_{0}^{\ell \left( Q\right)
}\left\vert e^{-t^{2}\mathcal{L}_{w}}\mathcal{L}_{w}\tilde{\chi}\tilde{F}%
_{Q}\left( x\right) \right\vert ^{2}t\, dt\, w\left( x\right) dx \\
& \qquad \leq \frac{4}{w\left( Q\right) }\int_{0}^{\ell \left( Q\right)
}\int_{5Q}\left\vert \mathcal{L}_{w}F_{Q}\left( x\right) \right\vert
^{2}w\left( x\right) dx\, t \, dt \\
&\qquad \leq C\ell \left( Q\right) ^{-2}\int_{0}^{\ell \left( Q\right) }t \,
dt \\
& \qquad \leq C<\infty.
\end{align*}

Therefore, to complete the proof we only have to show that (\ref{ineq:sqrf})
holds. Since $\int V_{t}\left( x,y\right) dy=0$, 
\begin{align*}
& R_{t}\left( \nabla \tilde{\chi}\tilde{F}_{Q}\right) \left( x\right) \\
&\qquad =\theta _{t}\mathbf{1}\left( x\right) \cdot P_{t}\left( \nabla 
\tilde{\chi}\tilde{F}_{Q}\right) \left( x\right) -\theta _{t}\left( \nabla 
\tilde{\chi}\tilde{F}_{Q}\right) \left( x\right) \\
&\qquad = -2te^{-t^{2}\mathcal{L}_{w}}\mathcal{L}_{w}\phi \left( x\right)
\cdot P_{t}\left( \nabla \tilde{\chi}\tilde{F}_{Q}\right) \left( x\right)
-2te^{-t^{2}\mathcal{L}_{w}}\mathcal{L}_{w}\tilde{\chi}\tilde{F}_{Q}\left(
x\right) \\
& \qquad =\left( \int V_{t}\left( x,y\right) y\,dy\right) P_{t}\left( \nabla 
\tilde{\chi}\tilde{F}_{Q}\right) \left( x\right) +\int V_{t}\left(
x,y\right) \left( \tilde{\chi}\tilde{F}_{Q}\right) \left( y\right) dy \\
&\qquad = \int V_{t}\left( x,y\right) \left[ \left( \tilde{\chi}\tilde{F}%
_{Q}\right) \left( y\right) -\left( \tilde{\chi}\tilde{F}_{Q}\right) \left(
x\right) -\left( y-x\right) P_{t}\left( \nabla \tilde{\chi}\tilde{F}%
_{Q}\right) \left( x\right) \right] dy \\
& \qquad = \int V_{t}\left( x,y\right) \tilde{G}_t(x,y)\,dy.
\end{align*}
By Lemma \ref{lemma-sq-Taylor}, 
\begin{multline}
\int_{\mathbb{R}^{n}}\int_{0}^{\infty }\left\vert R_{t}\left( \nabla \tilde{%
\chi}\tilde{F}_{Q}\right) \left( x\right) \right\vert ^{2}\frac{dt}{t}%
w\left( x\right) ~dx \\
=\int_{\mathbb{R}^{n}}\int_{0}^{\infty }\left\vert \int V_{t}\left(
x,y\right) \tilde{G}_{t}\left( \tilde{F}_{Q}\right) \left( x,y\right)
dy\right\vert ^{2}\frac{dt}{t}w\left( x\right) ~dx \leq C\left(
n,[w]_{A_{2}}\right) ^{2}\left\Vert {\nabla }\tilde{\chi}\tilde{F}%
_{Q}\right\Vert _{L^{2}(w)}^{2}.  \label{eqn:tb3}
\end{multline}
We estimate the right hand side using the product rule and the weighted
Poincar\'{e} inequality (see \cite{fabes-kenig-serapioni82}); if we fix $%
c_{0}=\frac{1}{w\left( 5Q\right) }\int_{5Q}F_{Q}\left( x\right) dx$, then 
\begin{align*}
\left\Vert {\nabla }\tilde{\chi}F_{Q}\right\Vert _{L^{2}(w)}^{2} & =\int
\left\vert {\nabla }\tilde{\chi}\tilde{F}_{Q}(x)\right\vert ^{2}w(x)\,dx \\
& \leq \frac{C}{\ell ^{2}\left( Q\right) }\int_{5Q}\left\vert
F_{Q}(x)-c_{0}\right\vert ^{2}w(x)\,dx +C\int_{5Q}\left\vert {\nabla }%
F_{Q}(x)\right\vert ^{2}w(x)\,dx \\
&\leq \frac{C\ell ^{2}\left( Q\right) }{\ell ^{2}\left( Q\right) }%
\int_{5Q}\left\vert \nabla F_{Q}(x)\right\vert ^{2}w(x)\,dx
+C\int_{5Q}\left\vert {\nabla }F_{Q}(x)\right\vert ^{2}w(x)\,dx \\
&\leq C\int_{5Q}\left\vert {\nabla }F_{Q}(x)\right\vert ^{2}w(x)\,dx.
\end{align*}
This proves (\ref{ineq:sqrf}) and our proof is complete.
\end{proof}

In our proof we actually need to replace (\ref{1:16iii}) with a more
complicated criterion; for ease of reference we record this as a separate
lemma.

\begin{lemma}
\label{lemma-TbRoot-mod} Suppose that there exists a finite index set $%
\left\{ \nu \right\} $ with cardinality $N=N\left( n,\lambda ,\Lambda
,w\right) $, such that for each cube $Q$ there are mappings $F_{Q,\nu
}:5Q\rightarrow \mathbb{C}$ and a constant $C=C\left( n,\lambda ,\Lambda
,w\right) <\infty $ satisfying%
\renewcommand{\theenumi}{\Roman{enumi}}%

\begin{enumerate}
\renewcommand{\itemsep}{6pt}

\item \label{2.23i}$\displaystyle \int_{5Q}\left\vert \nabla F_{Q,\nu
}\left( x\right) \right\vert ^{2}w\left( x\right) dx\leq C~w\left( Q\right) ,
$

\item \label{2.23ii}$\displaystyle \int_{10Q}\left\vert \mathcal{L}%
_{w}F_{Q,\nu }\left( x\right) \right\vert ^{2}w\left( x\right) dx\leq
Cw\left( Q\right) \ell \left( Q\right) ^{-2},$

\item \label{2.23iii}$\displaystyle{
\|\gamma_t\|_{C,w}^2 \leq C\sup_{t>0} \| \gamma _{t}\| _{\infty}^{2} +
C\sum_{\nu }\frac{1}{w( Q) }\int_{Q}\int_{0}^{\ell (Q) }| \gamma _{t}( x)
\nabla P_{t}F_{Q,\nu }(x) |^{2}\frac{dt}{t}\,w( x)\, dx. }$
\end{enumerate}

\smallskip

\noindent Then $\left\Vert \gamma _{t}\right\Vert _{C,w}\leq C\left(
n,\lambda ,\Lambda ,w\right) <\infty $: i.e., $\gamma _{t}$ is a Carleson
measure with respect to $w$.
\end{lemma}

The proof is essentially identical to the proof of Lemma~\ref{lemma-TbRoot}.
By (\ref{eqn-sup-gamma}), it is enough to show that 
\begin{equation*}
\sup_{Q}\frac{1}{w\left( Q\right) }\int_{Q}\int_{0}^{\ell \left( Q\right)
}\left\vert \gamma _{t}\left( x\right) \nabla P_{t}F_{Q,\nu }\left( x\right)
\right\vert ^{2}\frac{dt}{t}\,w( x)\, dx\leq C\left( n,\lambda ,\Lambda
,w\right) <\infty,
\end{equation*}
and this is done exactly as in the proof of (\ref{eqn:tb1}).

\section{Proof of the Weighted Kato Theorem}

\label{section:test-func}

To complete the proof of Theorem~\ref{theorem:WKato} we will construct a
finite index set $\{\nu\}$ and for each cube $Q\subset {\mathbb{R}^n}$ a
family of functions $F_{Q,\nu}: 5Q\rightarrow \mathbb{C}$ that satisfy the
hypotheses of Lemma~\ref{lemma-TbRoot-mod}. To do so we adapt the proof of
the non-weighted case in~\cite{hofmann01}.

Recall that $\varphi \left( x\right) =x$. Given a cube $Q\subset {\mathbb{R}%
^n}$ define the function 
\begin{equation}  \label{FQ}
F_{Q}\left( x\right) =e^{-\varepsilon ^{2}\ell \left( Q\right) ^{2}\mathcal{L%
}_{w}}\varphi \left( x\right),
\end{equation}
where $\varepsilon >0$ will be chosen below. The set $\{\nu\}$ will be a
finite collection of vectors in $\mathbb{C}^n$, $|\nu|=1$, also to be chosen
below. Define 
\begin{equation*}
F_{Q,\nu} = F_Q\cdot \bar{\nu}. 
\end{equation*}

Since $|F_{Q,\nu}|\leq |F_Q|$, and similarly for the gradient, to prove %
\eqref{2.23i} and \eqref{2.23ii} in Lemma~\ref{lemma-TbRoot-mod} it will
suffice to prove that $F_Q$ satisfies \eqref{1:16i} and \eqref{1:16ii} in
Lemma~\ref{lemma-TbRoot}. To prove \eqref{1:16ii}: from (\ref{eqn-sup-gamma}%
) we have%
\begin{equation*}
\left\vert t{\mathcal{L}}_{w}e^{-t^{2}{\mathcal{L}}_{w}}{\varphi }\left(
x\right) \right\vert =\left\vert V_{t}\varphi \left( x\right) \right\vert
=\left\vert \int V_{t}\left( x,y\right) y\,dy\right\vert \leq C<\infty
\end{equation*}%
for some constant independent of $t$. Then%
\begin{equation*}
\left\vert \mathcal{L}_{w}F_{Q}\left( x\right) \right\vert =\frac{1}{%
\varepsilon \ell \left( Q\right) }\left\vert V_{\varepsilon \ell \left(
Q\right) }\varphi \left( x\right) \right\vert =\frac{1}{\varepsilon \ell
\left( Q\right) }\left\vert \int V_{\varepsilon \ell \left( Q\right) }\left(
x,y\right) y\,dy\right\vert \leq \frac{C}{\varepsilon \ell \left( Q\right) }.
\end{equation*}
Since $w\in A_2$ it is a doubling measure, and so 
\begin{equation}
\int_{10Q}\left\vert \mathcal{L}_{w}F_{Q}\left( x\right) \right\vert
^{2}w\left( x\right) dx\leq \frac{C}{\varepsilon ^{2}}w\left( Q\right) \ell
\left( Q\right) ^{-2}.  \label{1.16ii-proved}
\end{equation}%
This proves (\ref{1:16ii}).

To prove \eqref{1:16i}, fix $\eta \in \mathcal{C}_{0}^{\infty }\left( 
\mathbb{R}^{n}\right) $ such that $\eta \equiv 1$ in $5Q$, $\eta \equiv 0$
in ${\mathbb{R}^n}\setminus 10Q$, $\left\Vert \eta \right\Vert _{\infty
}\leq 1$ and $\left\Vert \nabla \eta \right\Vert _{\infty }\leq C\ell \left(
Q\right) ^{-1}$. Then, by the ellipticity condition \eqref{eqn:degen}, 
\begin{multline}  \label{eqn:extra-step}
\int_{5Q}| {\nabla} F_{Q}(x)|^{2}w(x)\, dx \\
\leq \int_{\mathbb{R}^n}| {\nabla} F_{Q}(x)| ^{2}\eta(x)^{2}w(x) \,dx \leq
\lambda^{-1}\int_{\mathbb{R}^n}\mathbf{A}(x) {\nabla} F_{Q}(x) \cdot {\nabla}
F_{Q}(x) \eta(x)^{2}\,dx.
\end{multline}
By the conservation property, Lemma \ref{lemma:conservation}, $e^{-t^{2}%
\mathcal{L}_{w}}\mathbf{1}=\mathbf{1}$; hence, $\nabla _{x}e^{-t^{2}\mathcal{%
L}_{w}}\mathbf{1}=0$ and so we can write%
\begin{multline*}
{\nabla} _{x}e^{-t^{2}{\mathcal{L}}_w}\varphi (x) = {\nabla}_{x}\int
W_{t^{2}}\left( x,y\right) y\,dy \\
= {\nabla} _{x}\int W_{t^{2}}\left( x,y\right) ( y-x)\, dy+\mathbf{1} = {%
\nabla} _{x}G_{t}(x) +\mathbf{1},
\end{multline*}
where $G_{t}(x) =\big( e^{-t^{2}{\mathcal{L}}_w}-I\big)\varphi (x) $. Thus $%
\nabla F_{Q}=\mathbf{1}+\nabla G_{\varepsilon \ell \left( Q\right) }$. By
the Gaussian decay~(\ref{ineq:gauss2}) of the heat kernel, 
\begin{equation}
\left\vert G_{t}\left( x\right) \right\vert \leq \dfrac{C_{1}}{t^{n}}\dint
\exp \left( -C_{2}\dfrac{\left\vert x-y\right\vert ^{2}}{t^{2}}\right)
\left\vert y-x\right\vert dy\leq Ct.  \label{Gt}
\end{equation}
Integrating by parts, we have that 
\begin{align*}
& \int_{\mathbb{R}^n} \mathbf{A}(x) {\nabla} F_Q(x) \cdot {\nabla}
G_{\varepsilon\ell(Q)}(x)\eta(x)^2\,dx \\
& \qquad = \int_{\mathbb{R}^n} \mathbf{A}(x) {\nabla} F_Q(x) \cdot {\nabla}
(G_{\varepsilon\ell(Q)}\eta^2)(x)\,dx \\
& \qquad \qquad \qquad - 2\int_{\mathbb{R}^n} \mathbf{A}(x) {\nabla} F_Q(x)
\cdot G_{\varepsilon\ell(Q)}(x){\nabla}\eta(x)\eta(x)\,dx \\
& \qquad = \int_{\mathbb{R}^n} {\mathcal{L}}_w F_Q(x)
G_{\varepsilon\ell(Q)}(x)\eta(x)^2\,dx \\
& \qquad \qquad \qquad - 2\int_{\mathbb{R}^n} \mathbf{A}(x) {\nabla} F_Q(x)
\cdot G_{\varepsilon\ell(Q)}(x){\nabla}\eta(x)\eta(x)\,dx. \\
\end{align*}

If we combine this with \eqref{eqn:extra-step} we get that 
\begin{align*}
&\int_{\mathbb{R}^n}| {\nabla} F_{Q}(x) | ^{2}w(x) \eta(x)^{2}\,dx \\
& \qquad \leq \lambda ^{-1}\int_{\mathbb{R}^n}\mathbf{A}(x) {\nabla}
F_{Q}(x) \cdot \left( \mathbf{1}+{\nabla} G_{\varepsilon \ell(Q) }(x)
\right) \eta(x)^{2}dx \\
&\qquad \leq \lambda ^{-1}\int_{\mathbb{R}^n}| \mathbf{A}(x) {\nabla}
F_{Q}(x) \cdot \mathbf{1}| \eta(x)^{2}dx \\
& \qquad \qquad +\lambda ^{-1}\int_{\mathbb{R}^n}| {\mathcal{L}}_wF_{Q}(x) |
| G_{\varepsilon \ell (Q) }(x)|\eta(x)^{2}w(x)\, dx \\
& \qquad \qquad +2\lambda ^{-1}\int_{\mathbb{R}^n}\eta(x) | \mathbf{A}(x) {%
\nabla} F_{Q}(x) \cdot {\nabla} \eta(x) | | G_{\varepsilon \ell \left(
Q\right) }(x)|\, dx.
\end{align*}
By ellipticity, the inequality $| 2ab| \leq \delta a^{2}+\delta ^{-1}b^{2}$, 
$\delta >0$, and \eqref{Gt} we get 
\begin{align*}
& \int_{\mathbb{R}^n}| {\nabla} F_{Q}(x) | ^{2}w(x) \eta(x) ^{2}\,dx \\
& \qquad \leq \frac{n\delta \Lambda}{2\lambda } \int_{\mathbb{R}^n}| {\nabla}
F_{Q}(x) |^{2}w(x) \eta(x)^{2}dx + \frac{n\Lambda}{2 \lambda\delta }\int_{%
\mathbb{R}^n}\eta ^{2}w(x)\, dx \\
& \qquad \qquad +C\frac{\varepsilon ^{2}\ell(Q)^2}{\lambda} \int_{10Q}|{%
\mathcal{L}}_wF_{Q}(x)|^{2}w(x)\, dx + \frac{C}{\lambda}\int_{\mathbb{R}^n}%
\eta(x) ^{2}w(x) \,dx \\
&\qquad \qquad + \frac{\delta \Lambda}{ \lambda} \int_{\mathbb{R}^n}| {\nabla%
} F_{Q}(x)| ^{2}\eta(x) ^{2}w(x)\, dx +\frac{\Lambda }{\lambda\delta }%
\varepsilon ^{2} \ell(Q)^{2}\int_{10Q}| {\nabla} \eta(x) |^{2}w(x)\, dx.
\end{align*}
If we apply \eqref{1:16ii}, use the fact that $\varepsilon \leq 1$, $|{\nabla%
} \eta| \approx\ell(Q) ^{-1}$ and $w$ is doubling, and if we take $\delta>0$
sufficiently small, then 
\begin{equation*}
\int_{\mathbb{R}^n}| {\nabla} F_{Q}(x) | ^{2}w(x) \eta(x)^{2}\,dx \leq \frac{%
1}{2}\int_{\mathbb{R}^n}|{\nabla} F_{Q}(x)|^{2}w(x) \eta(x) ^{2}\,dx +C\frac{%
\Lambda }{\lambda}w( Q) .
\end{equation*}
Re-arranging terms and using the fact that $\eta \equiv 1$ on $5Q$, we get %
\eqref{1:16i}.

\medskip

To prove that \eqref{2.23iii} holds we make two reductions. First, recall
that the averaging operator $A_{t}$ is defined by 
\begin{equation*}
A_{t}f(x)=%
\mathchoice {{\setbox0=\hbox{$\displaystyle{\textstyle
-}{\int}$} \vcenter{\hbox{$\textstyle -$}}\kern-.5\wd0}} {{\setbox0=\hbox{$\textstyle{\scriptstyle -}{\int}$} \vcenter{\hbox{$\scriptstyle -$}}\kern-.5\wd0}} {{\setbox0=\hbox{$\scriptstyle{\scriptscriptstyle -}{\int}$} \vcenter{\hbox{$\scriptscriptstyle -$}}\kern-.5\wd0}} {{\setbox0=\hbox{$\scriptscriptstyle{\scriptscriptstyle -}{\int}$} \vcenter{\hbox{$\scriptscriptstyle -$}}\kern-.5\wd0}}%
\!\int_{Q_{t}(x)}f(y)\,dy,
\end{equation*}%
where $Q_{t}(x)$ is the unique dyadic cube containing $x$ such that $t\leq
\ell (Q_{t}(x))<2t$. As before, let $\tilde{\chi}\left( x\right) \in 
\mathcal{C}_{0}^{\infty }\left( \mathbb{R}^{n}\right) $ be such that ${%
\mathrm{supp}}(\tilde{\chi}) \subset 4Q$, $\left. \tilde{\chi}\right\vert
_{3Q}\equiv 1$ and $\left\vert \nabla \tilde{\chi}\right\vert \leq C\ell
\left( Q\right) ^{-1}$. Then $P_{t}f\left( x\right) =P_{t}\tilde{\chi}%
f\left( x\right) $, and $A_{t}f\left( x\right) =A_{t}\tilde{\chi}f\left(
x\right) $ for all $x\in Q$ and $t$, $0\leq t\leq \ell \left( Q\right) $.
Since $P_{t}$ commutes with the gradient, by Lemma \ref{lemma:dyadic-conv}
we have that 
\begin{align*}
&\int_{Q}\int_{0}^{\ell \left( Q\right) }\left\vert \gamma _{t}\left(
x\right) \nabla P_{t}F_{Q,\nu }\left( x\right) \right\vert ^{2}\frac{dt}{t}%
w\left( x\right) dx \\
& \qquad \leq \int_{Q}\int_{0}^{\ell \left( Q\right) }\left\vert \gamma
_{t}\left( x\right) \left( P_{t}-A_{t}\right) \tilde{\chi}\nabla F_{Q,\nu
}\left( x\right) \right\vert ^{2}\frac{dt}{t}w\left( x\right) dx \\
&\qquad \qquad +\int_{Q}\int_{0}^{\ell \left( Q\right) }\left\vert \gamma
_{t}\left( x\right) A_{t}\nabla F_{Q,\nu }\left( x\right) \right\vert ^{2}%
\frac{dt}{t}w\left( x\right) dx \\
&\qquad \leq \sup_{t>0}\left\Vert \gamma _{t}\right\Vert _{\infty
}^{2}\int_{Q}\int_{0}^{\ell \left( Q\right) }\left\vert \left(
P_{t}-A_{t}\right) \tilde{\chi}\nabla F_{Q,\nu }\left( x\right) \right\vert
^{2}\frac{dt}{t}w\left( x\right) dx \\
& \qquad \qquad +\int_{Q}\int_{0}^{\ell \left( Q\right) }\left\vert \gamma
_{t}\left( x\right) A_{t}\nabla F_{Q,\nu }\left( x\right) \right\vert ^{2}%
\frac{dt}{t}w\left( x\right) dx \\
&\qquad \leq C\sup_{t>0}\left\Vert \gamma _{t}\right\Vert _{\infty
}^{2}\int_{5Q}\left\vert \nabla F_{Q,\nu }\left( x\right) \right\vert
^{2}w\left( x\right) dx \\
& \qquad \qquad +\int_{Q}\int_{0}^{\ell \left( Q\right) }\left\vert \gamma
_{t}\left( x\right) A_{t}\nabla F_{Q,\nu }\left( x\right) \right\vert ^{2}%
\frac{dt}{t}w\left( x\right) dx.
\end{align*}
Therefore, by (\ref{2.23i}) we have 
\begin{multline*}
\frac{1}{w\left( Q\right) }\int_{Q}\int_{0}^{\ell \left( Q\right)
}\left\vert \gamma _{t}\left( x\right) \nabla P_{t}F_{Q,\nu }\left( x\right)
\right\vert ^{2}\frac{dt}{t}w\left( x\right) dx \\
\leq C\sup_{t>0}\left\Vert \gamma _{t}\right\Vert _{\infty }^{2}+\frac{1}{%
w\left( Q\right) }\int_{Q}\int_{0}^{\ell \left( Q\right) }\left\vert \gamma
_{t}\left( x\right) A_{t}\nabla F_{Q,\nu }\left( x\right) \right\vert ^{2}%
\frac{dt}{t}w\left( x\right) dx,
\end{multline*}
and so to prove (\ref{2.23iii}) it is enough to establish this estimate with 
$P_{t}$ replaced by $A_{t}$.

\smallskip

For our second reduction, we need a lemma that was proved in \cite{hofmann01}
in the unweighted case; essentially the same argument works in the weighted
case.

\begin{lemma}
\label{lemma-1.14} Suppose that there exist $0<\eta \leq 1$ such that for
all cubes $Q$, there exists a subset $E_{Q}\subset Q$ with $w\left(
E_{Q}\right) >\eta ~w\left( Q\right) $, and $Q\backslash E_{Q}=\bigcup Q_{j}$%
, where $\left\{ Q_{j}\right\}$ is a set of non-overlapping dyadic sub-cubes
of $Q$. Let $E_{Q}^{\ast}= R_{Q}\backslash \bigcup R_{Q_{j}}$, where $R_{Q}$
denotes the Carleson rectangle $Q\times \left( 0,\ell \left( Q\right)
\right) $ above $Q$. If for every cube $Q$, 
\begin{equation*}
\mu \left( E_{Q}^{\ast }\right) \leq C_{1}~w\left( Q\right),
\end{equation*}%
then $\mu $ is a Carleson measure with respect to $w$, and%
\begin{equation*}
\mu \left( Q^{\ast }\right) \leq \dfrac{C_{1}}{\eta }w\left( Q\right) .
\end{equation*}
\end{lemma}

Therefore, to complete our proof that \eqref{2.23iii} holds, we need to find
our finite index set $\{\nu\}$ and construct sets $E_Q$ as in Lemma~\ref%
{lemma-1.14} such that 
\begin{multline}  \label{eqn:finalstep}
\frac{1}{w(Q)}\iint_{E_Q^*} |\gamma_t(x)|^2\frac{dt}{t}\,w(x)\,dx \\
\leq C\sup_{t>0} \| \gamma _{t}\| _{\infty}^{2} + C\sum_{\nu }\frac{1}{w( Q) 
}\int_{Q}\int_{0}^{\ell (Q) }| \gamma _{t}( x) \nabla A_{t}F_{Q,\nu }(x)
|^{2}\,\frac{dt}{t}\,w( x)\, dx.
\end{multline}
We have now come to "the heart of the matter," as was said in \cite%
{hofmann01}. For every $\nu \in \mathbb{C}^{n} $ with $\left\vert \nu
\right\vert =1$ define the cone 
\begin{equation*}
\Gamma _{\nu }=\left\{ z\in \mathbb{C}^{n}:\left\vert z-\nu \left( z\cdot 
\bar{\nu}\right) \right\vert <\varepsilon \left\vert z\cdot \bar{\nu}%
\right\vert \right\} .
\end{equation*}
Clearly, for each $\varepsilon >0$ there exists a positive integer $%
N=N\left( \varepsilon ,n\right) $ and unit vectors $\nu _{1},\dots ,\nu
_{N}\in \mathbb{C}^{n}$ such that $\mathbb{C}^{n}\subset
\bigcup_{j=1}^{N}\Gamma _{\nu _{j}}$. Below we will fix our $\varepsilon$
and will then let $\{\nu\}=\{\nu_j\}_{j=1}^N$ be our index set.

We will first construct the sets $E_Q$. To do so, we will construct sets $%
E_{Q,\nu}$ that depend on $\nu$; the desired set $E_Q$ will be the union of
these sets over our index set $\{\nu\}$. We need two more lemmas. The first
is from \cite{hofmann01}; since its proof depends only on the Gaussian
bounds for the kernel of $e^{-t\mathcal{L}_{w}}$, it holds in the weighted
case without change.

\begin{lemma}
\label{lemma-2.1} There exists a constant $C>0$ depending only on the
constants $C_1,\,C_2$ in the Gaussian bounds, such that for any cube $%
Q\subset \mathbb{R}^{n}$, 
\begin{equation*}
\left\vert 
\mathchoice {{\setbox0=\hbox{$\displaystyle{\textstyle
-}{\int}$} \vcenter{\hbox{$\textstyle -$}}\kern-.5\wd0}} {{\setbox0=\hbox{$\textstyle{\scriptstyle -}{\int}$} \vcenter{\hbox{$\scriptstyle -$}}\kern-.5\wd0}} {{\setbox0=\hbox{$\scriptstyle{\scriptscriptstyle -}{\int}$} \vcenter{\hbox{$\scriptscriptstyle -$}}\kern-.5\wd0}} {{\setbox0=\hbox{$\scriptscriptstyle{\scriptscriptstyle -}{\int}$} \vcenter{\hbox{$\scriptscriptstyle -$}}\kern-.5\wd0}}%
\!\int_{Q}(\nabla F_{Q}(x)-\mathbf{1})\,dx\right\vert \leq C\varepsilon.
\end{equation*}
\end{lemma}

The second lemma gives two properties of $A_2$ weights. The first is the
well-known $A_\infty$ condition, and the second is closely related. For a
proof see~\cite{garciacuerva-rubiodefrancia85,grafakos04}.

\begin{lemma}
\label{lemma:Ainfty} If $w\in A_2$, there exists constants $\alpha,\,\delta>0
$ and constants $\beta,\,\epsilon>0$ such that given any cube $Q$ and
measurable set $E\subset Q$, 
\begin{equation*}
\frac{w(E)}{w(Q)} \leq \alpha\left(\frac{|E|}{|Q|}\right)^\delta \qquad 
\text{and} \qquad  \frac{|E|}{|Q|} \leq \beta\left(\frac{w(E)}{w(Q)}%
\right)^\epsilon. 
\end{equation*}
\end{lemma}

To prove \eqref{2.23iii} we first construct the sets $\{E_{Q,\nu }\}$ via
the stopping time argument used in \cite{hofmann01}. We include the details
to show how the proof adapts to the weighted case. Let $\mathcal{S}_{1}$ be
the collection of all maximal dyadic cubes $Q^{\prime }\subset Q$ such that 
\begin{equation}
\func{Re}%
\mathchoice {{\setbox0=\hbox{$\displaystyle{\textstyle
-}{\int}$} \vcenter{\hbox{$\textstyle -$}}\kern-.5\wd0}} {{\setbox0=\hbox{$\textstyle{\scriptstyle -}{\int}$} \vcenter{\hbox{$\scriptstyle -$}}\kern-.5\wd0}} {{\setbox0=\hbox{$\scriptstyle{\scriptscriptstyle -}{\int}$} \vcenter{\hbox{$\scriptscriptstyle -$}}\kern-.5\wd0}} {{\setbox0=\hbox{$\scriptscriptstyle{\scriptscriptstyle -}{\int}$} \vcenter{\hbox{$\scriptscriptstyle -$}}\kern-.5\wd0}}%
\!\int_{Q^{\prime }}\nu \cdot \nabla F_{Q,\nu }(y)\,dy\leq \frac{3}{4},
\label{eqn:2.7}
\end{equation}%
and let $\mathcal{B}^{1}=\mathcal{B}_{Q,\nu }^{1}=\bigcup_{\mathcal{S}%
_{1}}Q^{\prime }$. Similarly, let $\mathcal{S}_{2}$ be the collection of
maximal sub-cubes of $Q$ such that 
\begin{equation}
\mathchoice {{\setbox0=\hbox{$\displaystyle{\textstyle
-}{\int}$} \vcenter{\hbox{$\textstyle -$}}\kern-.5\wd0}} {{\setbox0=\hbox{$\textstyle{\scriptstyle -}{\int}$} \vcenter{\hbox{$\scriptstyle -$}}\kern-.5\wd0}} {{\setbox0=\hbox{$\scriptstyle{\scriptscriptstyle -}{\int}$} \vcenter{\hbox{$\scriptscriptstyle -$}}\kern-.5\wd0}} {{\setbox0=\hbox{$\scriptscriptstyle{\scriptscriptstyle -}{\int}$} \vcenter{\hbox{$\scriptscriptstyle -$}}\kern-.5\wd0}}%
\!\int_{Q^{\prime }}\left\vert \nabla F_{Q,\nu }(y)\right\vert \,dy>\frac{1}{%
8\varepsilon },  \label{eqn:2.8}
\end{equation}%
and let $\mathcal{B}^{2}=\mathcal{B}_{Q,\nu }^{2}=\bigcup_{\mathcal{S}%
_{2}}Q^{\prime }$. Set $\mathcal{B}=\mathcal{B}_{Q,\nu }=\mathcal{B}%
^{1}\bigcup \mathcal{B}^{2}$ and $E_{Q,\nu }=Q\backslash \mathcal{B}$. By
Lemma~\ref{lemma-2.1}, 
\begin{equation*}
\left\vert 
\mathchoice {{\setbox0=\hbox{$\displaystyle{\textstyle
-}{\int}$} \vcenter{\hbox{$\textstyle -$}}\kern-.5\wd0}} {{\setbox0=\hbox{$\textstyle{\scriptstyle -}{\int}$} \vcenter{\hbox{$\scriptstyle -$}}\kern-.5\wd0}} {{\setbox0=\hbox{$\scriptstyle{\scriptscriptstyle -}{\int}$} \vcenter{\hbox{$\scriptscriptstyle -$}}\kern-.5\wd0}} {{\setbox0=\hbox{$\scriptscriptstyle{\scriptscriptstyle -}{\int}$} \vcenter{\hbox{$\scriptscriptstyle -$}}\kern-.5\wd0}}%
\!\int_{Q}\left( \nu \cdot \nabla F_{Q,\nu }(y)\,dy-1\right) \right\vert
=\left\vert \nu \left( 
\mathchoice {{\setbox0=\hbox{$\displaystyle{\textstyle
-}{\int}$} \vcenter{\hbox{$\textstyle -$}}\kern-.5\wd0}} {{\setbox0=\hbox{$\textstyle{\scriptstyle -}{\int}$} \vcenter{\hbox{$\scriptstyle -$}}\kern-.5\wd0}} {{\setbox0=\hbox{$\scriptstyle{\scriptscriptstyle -}{\int}$} \vcenter{\hbox{$\scriptscriptstyle -$}}\kern-.5\wd0}} {{\setbox0=\hbox{$\scriptscriptstyle{\scriptscriptstyle -}{\int}$} \vcenter{\hbox{$\scriptscriptstyle -$}}\kern-.5\wd0}}%
\!\int_{Q}(\nabla F_{Q}(y)-\mathbf{1})\,dy\right) \cdot \bar{\nu}\right\vert
\leq C\varepsilon ,
\end{equation*}%
and thus 
\begin{multline}
\left( 1-C\varepsilon \right) \left\vert Q\right\vert \leq \func{Re}%
\int_{Q}\nu \cdot \nabla F_{Q,\nu }(y)\,dy  \label{eqn:gradF1} \\
=\func{Re}\left( \int_{E_{Q,\nu }}\nu \cdot \nabla F_{Q,\nu }(y)\,dy+\int_{%
\mathcal{B}^{1}}\nu \cdot \nabla F_{Q,\nu }(y)\,dy+\int_{\mathcal{B}%
\backslash \mathcal{B}^{1}}\nu \cdot \nabla F_{Q,\nu }(y)\,dy\right) .
\end{multline}%
By (\ref{eqn:2.7}) and the definition of $\mathcal{B}^{1},$%
\begin{equation}
\func{Re}\int_{\mathcal{B}^{1}}\nu \cdot \nabla F_{Q,\nu }(y)\,dy=\func{Re}%
\sum_{Q^{\prime }\in \mathcal{S}_{1}}\int_{Q^{\prime }}\nu \cdot \nabla
F_{Q,\nu }(y)\,dy\leq \frac{3}{4}\sum_{Q^{\prime }\in \mathcal{S}%
_{1}}\left\vert Q^{\prime }\right\vert \leq \frac{3}{4}\left\vert
Q\right\vert .  \label{eqn:B1}
\end{equation}%
By H\"{o}lder's inequality, \eqref{2.23i} and the definition of $A_{2}$
weights, 
\begin{multline}
\func{Re}\int_{\mathcal{B}\backslash \mathcal{B}^{1}}\nu \cdot \nabla
F_{Q,\nu }(y)\,dy \\
\leq \left( \int_{\mathcal{B}^{2}}\left\vert \nabla F_{Q,\nu }(y)\right\vert
^{2}w(y)\,dy\right) ^{1/2}w^{-1}\left( \mathcal{B}^{2}\right) ^{1/2}\leq
Cw\left( Q\right) ^{1/2}w^{-1}\left( \mathcal{B}^{2}\right) ^{1/2} \\
\leq C\left( w^{-1}\left( Q\right) w\left( Q\right) \right) ^{1/2}\left( 
\frac{w^{-1}\left( \mathcal{B}^{2}\right) }{w^{-1}\left( Q\right) }\right)
^{1/2}\leq C\left[ w\right] _{A_{2}}\left\vert Q\right\vert \left( \frac{%
w^{-1}\left( \mathcal{B}^{2}\right) }{w^{-1}\left( Q\right) }\right) ^{1/2}.
\label{eqn:B2}
\end{multline}

Similarly, by the definition of $\mathcal{B}^{2}$ and (\ref{2.23i}), 
\begin{multline}
\left\vert \mathcal{B}^{2}\right\vert \leq \sum_{\mathcal{S}_{2}}\left\vert
Q^{\prime }\right\vert <8\varepsilon \sum_{\mathcal{S}_{2}}\int_{Q^{\prime
}}\left\vert \nabla F_{Q,\nu }(y)\right\vert \,dy=8\varepsilon \int_{%
\mathcal{B}^{2}}\left\vert \nabla F_{Q,\nu }(y)\right\vert \,dy
\label{unwtd-bound} \\
\leq 8\varepsilon \left( \int_{\mathcal{B}^{2}}\left\vert \nabla F_{Q,\nu
}(y)\right\vert ^{2}w(y)\,dy\right) ^{1/2}w^{-1}\left( \mathcal{B}%
^{2}\right) ^{1/2} \\
\leq 8C\varepsilon w\left( Q\right) ^{1/2}w^{-1}\left( Q\right) ^{1/2}\leq
8C\varepsilon \left[ w\right] _{A_{2}}\left\vert Q\right\vert =C\varepsilon
\left\vert Q\right\vert .
\end{multline}%
By Lemma~\ref{lemma:Ainfty} we may assume that $\varepsilon $ is so small
that \eqref{unwtd-bound} implies 
\begin{equation*}
\frac{w^{-1}\left( \mathcal{B}^{2}\right) }{w^{-1}\left( Q\right) }\leq
\left( \frac{1/8}{1+C\left[ w\right] _{A_{2}}}\right) ^{2}.
\end{equation*}%
Combining this estimate with (\ref{eqn:B2}) we obtain 
\begin{equation*}
\func{Re}\int_{\mathcal{B}\backslash \mathcal{B}^{1}}\nu \cdot \nabla
F_{Q,\nu }(y)\,dy\leq \frac{1}{8}\left\vert Q\right\vert ,
\end{equation*}%
and putting this together with inequalities (\ref{eqn:gradF1}) and (\ref%
{eqn:B1}), for $\varepsilon $ small enough we get 
\begin{equation*}
\frac{1}{16}\left\vert Q\right\vert \leq \func{Re}\left( \int_{E_{Q,\nu
}}\nu \cdot \nabla F_{Q,\nu }(y)\,dy\right) .
\end{equation*}%
But then, since $w\in A_{2},$ 
\begin{multline*}
\frac{1}{16}\left\vert Q\right\vert \leq \func{Re}\left( \int_{E_{Q,\nu
}}\nu \cdot \nabla F_{Q,\nu }(y)\,dy\right)  \\
\leq C\left( w\left( Q\right) w^{-1}\left( E_{Q,\nu }\right) \right)
^{1/2}\leq C\left\vert Q\right\vert \left( \frac{w^{-1}\left( E_{Q,\nu
}\right) }{w^{-1}\left( Q\right) }\right) ^{1/2}.
\end{multline*}%
Since $w^{-1}$ is also in $A_{2}$, by Lemma~\ref{lemma:Ainfty} (applied
twice) there exists $\eta >0$ such that 
\begin{equation}
w\left( E_{Q,\nu }\right) \geq \eta ~w\left( Q\right) .  \label{eqn:etaw}
\end{equation}

\bigskip

Now write $\mathcal{B}_{Q,\nu }=Q\backslash E_{Q_{,\nu }}=\bigcup Q_{\nu ,j}$%
, where the $Q_{\nu j}$ are disjoint maximal dyadic cubes. Let $E_{Q,\nu
}^{\ast }=R_{Q}\setminus \bigcup R_{Q_{\nu ,j}}$ be the sawtooth region
above $E_{Q,\nu }$. If $\left( x,t\right) \in E_{Q,\nu }^{\ast }$, and $%
Q_{t}\left( x\right) $ is the biggest dyadic sub-cube of $Q$ containing $x$
with $t\leq \ell \left( Q_{t}\left( x\right) \right) <2t$, then by the
maximality of the cubes in $\mathcal{S}_{1}$ and $\mathcal{S}_{2}$ we have
that $Q_{t}\left( x\right) \not\in \mathcal{B}$; hence 
\begin{equation*}
\frac{3}{4}<\func{Re}\mathchoice{{\setbox0=\hbox{$\displaystyle{\textstyle
-}{\int}$}\vcenter{\hbox{$\textstyle -$}}\kern-.5\wd0}}{{\setbox0=%
\hbox{$\textstyle{\scriptstyle -}{\int}$}\vcenter{\hbox{$\scriptstyle -$}}%
\kern-.5\wd0}}{{\setbox0=\hbox{$\scriptstyle{\scriptscriptstyle -}{\int}$}%
\vcenter{\hbox{$\scriptscriptstyle -$}}\kern-.5\wd0}}{{\setbox0=%
\hbox{$\scriptscriptstyle{\scriptscriptstyle -}{\int}$}\vcenter{\hbox{$%
\scriptscriptstyle -$}}\kern-.5\wd0}}\!\int_{Q_{t}\left( x\right) }\nu \cdot
\nabla F_{Q,\nu }(y)\,dy\leq \mathchoice{{\setbox0=\hbox{$\displaystyle{%
\textstyle -}{\int}$}\vcenter{\hbox{$\textstyle -$}}\kern-.5\wd0}}{{\setbox0=%
\hbox{$\textstyle{\scriptstyle -}{\int}$}\vcenter{\hbox{$\scriptstyle -$}}%
\kern-.5\wd0}}{{\setbox0=\hbox{$\scriptstyle{\scriptscriptstyle -}{\int}$}%
\vcenter{\hbox{$\scriptscriptstyle -$}}\kern-.5\wd0}}{{\setbox0=%
\hbox{$\scriptscriptstyle{\scriptscriptstyle -}{\int}$}\vcenter{\hbox{$%
\scriptscriptstyle -$}}\kern-.5\wd0}}\!\int_{Q_{t}\left( x\right)
}\left\vert \nabla F_{Q,\nu }(y)\right\vert \,dy\leq \frac{1}{8\varepsilon }.
\end{equation*}%
By the definition of $A_{t}$, this implies that 
\begin{equation*}
\frac{3}{4}<\func{Re}\nu \cdot A_{t}\nabla F_{Q,\nu }\left( x\right) \leq
\left\vert A_{t}\nabla F_{Q,\nu }\left( x\right) \right\vert \leq \frac{1}{%
8\varepsilon }.
\end{equation*}%
By the definition of $\Gamma _{\nu }$, if $z\in \Gamma _{\nu }$, then $%
|z|<(1+\varepsilon )|z\cdot \bar{\nu}|$. Hence, for $\varepsilon >0$
sufficiently small, we have that 
\begin{multline*}
|z\cdot A_{t}F_{Q,\nu }(x)|\geq |(z\cdot \bar{\nu})(\nu \cdot A_{t}F_{Q,\nu
}(x))|-|(z-\nu (z\cdot \bar{\nu})\cdot A_{t}F_{Q,\nu }(x)| \\
\geq \frac{3}{4}|z\cdot \bar{\nu}|-\varepsilon |z\cdot \bar{\nu}%
||A_{t}F_{Q,\nu }(x)|\geq \frac{5}{8}|z\cdot \bar{\nu}|\geq \frac{1}{2}|z|.
\end{multline*}

Now fix $\varepsilon $ small; form our index set $\{\nu \}=\{\nu
_{j}\}_{j=1}^{N}$ as described above, and let $E_{Q}=\bigcup_{j}E_{Q,\nu
_{j}}$. Therefore, if we let $z=\gamma _{t}(x)$, then $\gamma _{t}\left( x\right)
\in \Gamma _{\nu _{j}}$ for some $1\leq j\leq N$, and we have 
\begin{equation*}
\frac{1}{4}\left\vert \gamma _{t}(x)\right\vert ^{2}\leq
\sum_{j=1}^{N}\left\vert \gamma _{t}(x)\cdot A_{t}\nabla F_{Q,\nu
_{j}}\left( x\right) \right\vert ^{2},
\end{equation*}%
for all $(x,t)\in E_{Q}^{\ast }$. It follows that 
\begin{align*}
& \dfrac{1}{w(Q)}\iint_{E_{Q}^{\ast }}|\gamma _{t}(x)|^{2}\,\dfrac{dt}{t}%
\,w(x)\,dx \\
& \qquad \leq \dfrac{4}{w(Q)}\sum_{j=1}^{N}\iint_{E_{Q}^{\ast }}|\gamma
_{t}(x)\cdot A_{t}\nabla F_{Q,\nu }(x)|^{2}\,\dfrac{dt}{t}\,w(x)\,dx \\
& \qquad \leq \dfrac{4}{w(Q)}\sum_{j=1}^{N}\int_{Q}\int_{0}^{\ell
(Q)}|\gamma _{t}(x)\cdot A_{t}\nabla F_{Q,\nu }(x)|^{2}\,\dfrac{dt}{t}%
\,w(x)\,dx,
\end{align*}%
where, by \eqref{eqn:etaw}, $w(E_{Q})\geq \eta \,w(Q)$. 
This proves \eqref{eqn:finalstep}; thus we have shown that (\ref{2.23iii})
in Lemma~\ref{lemma-TbRoot-mod} holds, and so have completed the proof of
Theorem \ref{theorem:WKato}.

\bibliographystyle{plain}
\bibliography{kato}

\end{document}